\documentclass[letterpaper, reqno,11pt]{article}
\usepackage[margin=1.0in]{geometry}
\usepackage{color,latexsym,amsmath,amssymb, amsthm, graphicx,subfigure,revsymb, enumerate}
\usepackage[percent]{overpic}

\numberwithin{equation}{section}

\newcommand{\RR}{\mathbb{R}}

\newcommand{\ZZ}{\mathbb{Z}}

\newcommand{\eps}{\varepsilon}
\newcommand{\tubes}{\mathbb{T}}

\newcommand{\lines}{\mathcal L}

\newcommand{\norm}[1]{\left\Vert#1\right\Vert}
\newcommand{\itemizeEqnVSpacing}{\rule{0pt}{1pt}\vspace*{-12pt}}

\newtheorem{thm}{Theorem}[section]
\newtheorem{lem}[thm]{Lemma}
\newtheorem{prop}[thm]{Proposition}
\newtheorem{cor}[thm]{Corollary}
\newtheorem{conj}[thm]{Conjecture}

\theoremstyle{remark}
\newtheorem{defn}[thm]{Definition}

\newtheorem{rem}[thm]{Remark}

\newtheorem*{kLinearKakeyaThmRestate}{Theorem \ref{kLinearKakeyaThm}$^\prime$}

\makeatletter
\newcommand{\myitem}[1]{%
\item[#1.]\protected@edef\@currentlabel{#1}%
}
\makeatother

\providecommand{\keywords}[1]{\textbf{\textit{Keywords$\phantom{--}$ }} #1}

\begin{document}
\pagenumbering{arabic}
\title{New Kakeya estimates using Gromov's algebraic lemma}

\author{Joshua Zahl\thanks{University of British Columbia, Vancouver BC, supported by an NSERC Discovery grant, jzahl@math.ubc.ca.
}}

\maketitle

\begin{abstract}
\noindent This paper presents several new results related to the Kakeya problem. First, we establish a geometric inequality which says that collections of direction-separated tubes (thin neighborhoods of line segments that point in different directions) cannot cluster inside thin neighborhoods of low degree algebraic varieties. We use this geometric inequality to obtain a new family of multilinear Kakeya estimates for direction-separated tubes. Using the linear / multilinear theory of Bourgain and Guth, these multilinear Kakeya estimates are converted into Kakeya maximal function estimates. Specifically, we obtain a Kakeya maximal function estimate in $\RR^n$ at dimension $d(n) = (2-\sqrt{2})n + c(n)$ for some $c(n)>0$. Our bounds are new in all dimensions except $n=2,3,4,$ and $6$.
\end{abstract}


\section{Introduction}
A set $T\subset\RR^n$ is called a $\lambda\times\delta$ tube if it is a translated and rotated copy of the set 
$$
\big\{x\in\RR^n\colon \sqrt{x_1^2+\ldots+x_{n-1}^2}<\delta,\ 0< x_n<\lambda\big\}.
$$ 
Every $\lambda\times\delta$ tube has a unique coaxial line, and we say that the angle between two tubes is the angle between their coaxial lines. We say a set of $\lambda\times\delta$ tubes is direction-separated if the angle between each pair of tubes is at least $\delta$. In this paper we will be interested in the Kakeya maximal function conjecture, which is a quantitative bound on the overlap between direction-separated tubes.

\begin{conj}[Kakeya maximal function conjecture]\label{kakeyaMaximalFunction}
Let $1\leq d\leq n$ and let $\eps>0$. Then there is a constant $C(n,d,\eps)$ so that whenever $\tubes$ is a set of direction-separated $1\times\delta$ tubes in $\RR^n$, we have
\begin{equation}\label{kakeyaMaximalFunctionConjecture}
\Big\Vert \sum_{T\in\tubes}\chi_T\Big\Vert_{\frac{d}{d-1}}\leq C(n,d,\eps)\Big(\frac{1}{\delta}\Big)^{\frac{n}{d}-1+\eps}\Big(\sum_{T\in\tubes}|T|\Big)^{\frac{d-1}{d}}.
\end{equation} 
\end{conj}
If Conjecture \ref{kakeyaMaximalFunction} is true for a particular value of $d$ and $n$, this is called a Kakeya maximal function estimate in $\RR^n$ at dimension $d$. A Kakeya maximal function estimate in $\RR^n$ at dimension $d$ implies that every Besicovitch set in $\RR^n$ has Hausdorff dimension at least $d$. Further background on the Kakeya conjecture can be found in the survey articles \cite{KT,W2}.

Conjecture \ref{kakeyaMaximalFunction} was solved in dimension 2 by C\'ordoba \cite{Cor}, and remains open in dimension three and higher. In 2005, Bennett, Carbery, and Tao \cite{BCT} considered the following multilinear variant of the Kakeya maximal function conjecture.
\begin{thm}[Multilinear Kakeya theorem]\label{multilinearKakeyaThm}
Let $2\leq k\leq n$. Then there is a constant $C(n)$ so that whenever $\tubes_1,\ldots,\tubes_k$ are sets of $1\times\delta$ tubes in $\RR^n$, we have
\begin{equation}
\Big\Vert\Big(\sum_{T_1\in\tubes_1}\ldots\sum_{T_k\in\tubes_k} \chi_{T_1}\cdots\chi_{T_k}|v_1\wedge\ldots\wedge v_k|\Big)^{\frac{1}{k}}\Big\Vert_{\frac{k}{k-1}}
\leq C(n) \Big(\frac{1}{\delta}\Big)^{\frac{n}{k}-1}\prod_{i=1}^k\Big(\sum_{T_i\in\tubes_i}|T_i|\Big)^{\frac{1}{k}},
\end{equation}
where in the above expression $v_i=v(T_i)$ is the direction of the tube $T_i$. 
\end{thm}
Theorem \ref{multilinearKakeyaThm} was proved up to endpoint by Bennett, Carbery, and Tao in \cite{BCT} and the endpoint estimate was established by Guth in \cite{G}. The version stated here can be found in \cite{CV}. Theorem \ref{multilinearKakeyaThm} can be combined with an induction-on-scales argument to prove bounds on the Kakeya maximal function. This is discussed further in Lemma \ref{multiLinToLin} below. Unfortunately, this strategy is not particularly effective---the resulting estimates are worse than those coming from other methods, such as Wolff's hairbrush argument \cite{W1}. What's more, Theorem \ref{multilinearKakeyaThm} is sharp, so this suggests that the strategy of combining induction-on-scales with $k$-linear Kakeya estimates will not lead to new bounds on the Kakeya maximal function.

However, the Kakeya maximal function conjecture is a statement about direction-separated tubes, while Theorem \ref{multilinearKakeyaThm} does not impose this restriction. Thus while Theorem \ref{multilinearKakeyaThm} is sharp, stronger estimates are possible if we impose the additional requirement that the tubes are direction-separated. 
\begin{thm}[Direction-separated multilinear Kakeya]\label{kLinearKakeyaThm}
Let $2\leq k\leq n$ and let $\eps>0$. Then there is a constant $C(n,\eps)$ so that whenever $\tubes$ is a set of direction-separated $1\times\delta$ tubes in $\RR^n$, we have
\begin{equation}\label{kLinRd}
\Big\Vert\Big(\sum_{T_1,\ldots,T_k\in\tubes} \chi_{T_1}\cdots\chi_{T_k}|v_1\wedge\ldots\wedge v_k|^{\frac{k}{d}}\Big)^{\frac{1}{k}} \Big\Vert_{\frac{d}{d-1}} \leq C(n,\eps) \Big(\frac{1}{\delta}\Big)^{\frac{n}{d}-1+\eps}\Big(\sum_{T\in\tubes}|T|\Big)^{\frac{n(d-1)}{(n-1)d}},
\end{equation}
where
\begin{equation}\label{valueOfDmultiLinKakeya}
d = \frac{n^2+k^2+n-k}{2n}.
\end{equation}
\end{thm}

Note that when $k<n$, the value of $d$ from \eqref{valueOfDmultiLinKakeya} is larger than $k$. Theorem \ref{kLinearKakeyaThm} generalizes a previous result of Guth and the author \cite{GZ}, which proved\footnote{The chronology is actually a bit more complicated; \cite{GZ} proved a result that was conditional on a certain conjecture about tubes. The author \cite{Z}, and independently, Katz and Rogers \cite{KR} later provided this conjecture.} Theorem \ref{kLinearKakeyaThm} in the special case $d=3,n=4$. The techniques in \cite{GZ} naturally extend to the case $k=n-1$, and they can also be used to prove weaker variants of Theorem \ref{kLinRd} for general $n$ and $k$ (Hickman and Rogers \cite{HR2} employed a similar strategy to prove certain $k$-broad estimates in $\RR^n$). However, several additional ideas are needed when $k<n-1$. Chief among these is a new hierarchical ``grains'' decomposition, which is constructed in Section \ref{multiLevelGrainsSec}, and a new geometric inequality that controls how tubes can interact with this grains decomposition; this geometric inequality will be discussed further in Section \ref{insideThickenedVarieties}. We will prove Theorem \ref{kLinearKakeyaThm} in Section \ref{multilinearKakeyaSec}.

When $k=n-1$ then $d = n-1+1/n$, and Theorem \ref{kLinearKakeyaThm} is sharp. This can be seen by taking $\tubes$ to be a set of $\delta^{2-n}$ direction-separated tubes that are contained in a rectangular prism in $\RR^n$ of dimensions $\delta\times 1\times\ldots\times 1$. However, if Conjecture \ref{kakeyaMaximalFunction} is true for a particular value of $n$ and $d$, with $2\leq d\leq n$, then this immediately implies that a slightly weaker variant of \eqref{kLinRd} is also true for this value of $n$ and $d$, where the exponent $n(d-1)/(n-1)d$ is replaced by $(d-1)/d$. The next lemma partially reverses this implication; it says that under certain restrictions, bounds of the form \eqref{kLinRd} imply bounds on the Kakeya maximal function.
\begin{lem}[Multilinear to linear Kakeya]\label{multiLinToLin}
Let $2\leq k\leq n$ and let $d\leq n-k+2$. Suppose that for each $\eps>0$, there is a constant $C(n,\eps)$ so that the inequality
$$
\Big\Vert\Big(\sum_{T_1,\ldots,T_k\in\tubes} \chi_{T_1}\cdots\chi_{T_k}|v_1\wedge\ldots\wedge v_k|^{\frac{k}{d}}\Big)^{\frac{1}{k}} \Big\Vert_{\frac{d}{d-1}} \leq C(n,\eps) \Big(\frac{1}{\delta}\Big)^{\frac{n}{d}-1+\eps}\Big(\sum_{T\in\tubes}|T|\Big)^{\frac{d-1}{d}}
$$
holds for all sets $\tubes$ of direction-separated $1\times\delta$ tubes. 

Then for each $\eps>0$, there is a constant $C^\prime(n,\eps)$ so that the inequality
\begin{equation}\label{sharpPQ}
\Big\Vert \sum_{T\in\tubes}\chi_T\Big\Vert_{\frac{d}{d-1}}\leq C^\prime(n,\eps)\Big(\frac{1}{\delta}\Big)^{\frac{n}{d}-1+\eps}\Big(\sum_{T\in\tubes}|T|\Big)^{\frac{d-1}{d}}
\end{equation}
holds for all sets $\tubes$ of direction-separated $1\times\delta$ tubes.
\end{lem}

Lemma \ref{multiLinToLin} is standard. A variant of the lemma was proved by Bourgain and Guth \cite{BG} in the context of the restriction problem, and a version similar to the one stated here can be found in \cite{HR2}. Combining Theorem \ref{kLinearKakeyaThm} and Lemma \ref{multiLinToLin}, we obtain the following bounds on the Kakeya maximal function.
\begin{thm}\label{KakeyaMaximalBoundThm}
For each integer $n\geq 2$, a Kakeya maximal function estimate in $\RR^n$ holds at dimension 
\begin{equation}\label{defnDn}
d(n)=\max_{2\leq k\leq n} \min\Big(n-k+2, \frac{n^2+k^2+n-k}{2n}\Big).
\end{equation} 
\end{thm} 
Note that the minimum in \eqref{defnDn} occurs when $k$ is the floor or ceiling of 
$$
\sqrt{2n^2+2n+\frac{1}{4}}-n+1/2=(\sqrt 2-1)n+\frac{1+\sqrt{2}}{2}+O(1/n).
$$ 
In particular, $d(n) \geq (2-\sqrt 2)n$. The Kakeya maximal function estimate from Theorem \ref{KakeyaMaximalBoundThm} is new in all dimensions except $n=2,3,4,6$. It also implies new bounds on the Hausdorff dimension of Kakeya sets in $\RR^n$ for certain (but not all) values of $n$. The table below shows the current best known bounds for $d(n)$
\begin{center}
\begin{tabular}{c|c|l||c|c|l}
$n$ & $d(n)$ & &  $n$ & $d(n)$\\
\hline
2 & 2 & C\'ordoba \cite{Cor} & 					6 & 4 & Wolff \cite{W1}\\
3 & $2.5+\eps_0$ & Katz-Zahl \cite{KZ,KZ2}& 	7 & 4.857 & Theorem \ref{KakeyaMaximalBoundThm}\\
4 & $3.059$ & Katz-Zahl \cite{KZ2} &			8 & 5.25 & Theorem \ref{KakeyaMaximalBoundThm} \\
5 & 3.6 & Theorem \ref{KakeyaMaximalBoundThm}&	$\smash{\vdots}$ & $\smash{\vdots}$ & $\phantom{1111}$ $\smash{\vdots}$  \\
\end{tabular}
\end{center}
In high dimensions, the previous best-known bound on the Kakeya maximal function was $d(n)=(4n+3)/7$, due to Katz and Tao \cite{KT}. In certain intermediate dimensions $5\leq n\leq 100$ the previous best-known bound was due to Hickman and Rogers \cite{HR2}. 

\medskip

\begin{rem}
Recently, the author became aware that Hickman, Rogers, and Zhang have concurrently and independently proved Theorem \ref{KakeyaMaximalBoundThm}. They did this by proving a nearly identical version of Theorem \ref{polyWolffFlagsVarities} (stated below), using similar arguments. Instead of proving Theorem \ref{kLinearKakeyaThm}, they established what is known as a $k$-broad estimate. While $k$-broad estimates are slightly weaker than the corresponding $k$-linear estimates, they are nonetheless sufficient to use (a variant of) the multilinear to linear Kakeya argument from Lemma \ref{multiLinToLin}.
\end{rem}

\subsection{Tubes inside thickened algebraic varieties}\label{insideThickenedVarieties}
An important new ingredient used in the proof of Theorem \ref{kLinearKakeyaThm} is a new geometric inequality that bounds the number of direction-separated tubes that can cluster near a nested sequence of low-degree varieties. We will discuss this further below.

A set $S\subset \RR^n$ is called \emph{semi-algebraic} if it can be written as a finite union of sets of the form 
$$
\{x\in\RR^n\colon P_1(x)>0,\ldots,P_k(x)>0,P_{k+1}(x)=0,\ldots,P_{k+\ell}(x)=0\},
$$
where $P_1,\ldots,P_{k+\ell}$ are polynomials. A union of such sets is called a presentation of $S$. The complexity of a presentation is the sum of the degrees of the polynomials involved (with multiplicities). The complexity of a semi-algebraic set $S$ is the minimum complexity of its presentations.

In \cite{KR}, Katz and Rogers resolved a conjecture of Guth \cite{G3} and Guth-Zahl \cite{GZ} concerning the number of direction-separated tubes that can have large intersection with a semi-algebraic set. 
\begin{thm}[Direction-separated tubes obey the polynomial Wolff axioms]\label{polyWolffAxiomsThm}
Let $n$ and $E$ be integers, with $n\geq 2$, and let $\eps>0$. Then there is a constant $C(n,E,\eps)$ so that for every semi-algebraic set $S\subset\RR^n$ of complexity at most $E$ and for every set $\tubes$ of direction-separated $1\times \delta$ tubes, we have
\begin{equation}\label{polyWolffAxiomsIneq}
\#\{T\in\tubes\colon |T\cap S|\geq r|T|\}\leq C(n,E,\eps)|S|\delta^{1-n-\eps}r^{-n}.
\end{equation}
\end{thm}
A particularly interesting example occurs when $S$ is a thin neighborhood of an algebraic variety and $r$ is comparable to the diameter of $S$.
\begin{cor}\label{polyWolffAxiomsCor}
Let $n,d,$ and $E$ be integers with $1\leq d<n$, and let $\eps>0$. Then there is a constant $C(n,E,\eps)>0$ so that for every algebraic variety $Z\subset\RR^n$ of codimension $d$ that is defined by polynomials of degree at most $E$, for every set $\tubes$ of direction-separated $1\times \delta$ tubes, and for every $x\in\RR^n$, we have
\begin{equation}\label{polyWolffAxiomsVarietyIneq}
\#\{T\in\tubes\colon |T\cap N_{2\delta}(Z)\cap B(x,r)|\geq r|T|\}\leq C(n,E,\eps) \frac{\delta^{d+1-n-\eps}}{r^d}.
\end{equation}
\end{cor}

Corollary \ref{polyWolffAxiomsCor} was used by Guth and the author in \cite{GZ}\footnote{Actually, the paper \cite{GZ} preceded \cite{KR}; it proved a conditional bound that relied on a special case of Corollary \ref{polyWolffAxiomsCor}. This special case was proved by the author in \cite{Z}, and then shortly afterward the general case was proved by Katz and Rogers in \cite{KR}.} and by Katz and the author in \cite{KZ2} to obtain improved bounds on the Kakeya maximal function in $\RR^4$. It was used by Hickman and Rogers \cite{HR1,HR2} to obtain improved Kakeya bounds for certain dimensions $n\geq 5$, and to obtain improved restriction estimates in dimension $\geq 13$, as well as dimension $4,5,7,9,$ and $11.$ 

We will prove the following generalization of Corollary \ref{polyWolffAxiomsCor}.

\begin{thm}[Direction-separated tubes and sequences varieties]\label{polyWolffFlagsVarities}
Let $n$ and $E$ be integers with $n\geq 2$, and let $\eps>0$. Then there is a constant $C(n,E,\eps)>0$ so that the following holds. Let $Z_1\supset \cdots\supset Z_d$ be a nested sequence of real algebraic varieties, each defined by polynomials of degree at most $E$. Suppose that $Z_i$ has codimension at least $i$. Let $1\geq r_1\geq\ldots\geq r_d\geq\delta$. Let $\tubes$ be a set of direction-separated $1\times \delta$ tubes and let $x\in\RR^n$. Then
\begin{equation}\label{boundNumberDirectionSeparatedTubesVarietyVersion}
\#\{T\in\tubes\colon |T\cap N_{2\delta}(Z_i) \cap B(x,r_i)| \geq r_i|T|,\ i=1,\ldots,d\}
\leq C(n,E,\eps) \frac{\delta^{d+1-n-\eps}}{r_1 \cdots r_d}.
\end{equation}
\end{thm}
We can see that Theorem \ref{polyWolffFlagsVarities} is indeed a generalization of Corollary \ref{polyWolffAxiomsCor} by taking $Z_i=Z$ and $r_i=r$ for each $i=1,\ldots,d$. Theorem \ref{polyWolffFlagsVarities} will be proved in Section \ref{polyWolffNestedSequences} below. 

\subsection{Notation}
If $X\subset\RR^n$, we will use $|X|$ to denote the Lebesgue measure of $X$, and we will use $\#X$ to denote the cardinality of $X$. If $\rho>0$, we will write $N_{\rho}(X)$ to denote the $\rho$-neighborhood of $X$, and $\mathcal{E}_{\rho}(X)$ to denote the $\rho$-covering number of $X$. Finally, we will define $\operatorname{CC}(X)$ to be the set of (Euclidean) connected components of $X$.

If $T\subset\RR^n$ is a $\lambda\times\delta$ tube, we will define $v(T)$ to be the unit vector parallel to the line $L$ coaxial with $T$. Note that both $v(T)$ and $-v(T)$ are parallel to $L$. For concreteness, we will select $v(T)=(v_1,\ldots,v_n)$ so that the last non-zero coordinate is positive.


We will write $A\lesssim B$ or $A = O(B)$ to mean there exists a constant $C$, depending only on the ambient dimension $n$, so that $A\leq CB$. If the constant $C$ is allowed to depend on additional parameter, such as $\eps$, then we will write $A\lesssim_{\eps}B$. To simplify the statement of results such as Theorem \ref{kLinearKakeyaThm} and Theorem \ref{KakeyaMaximalBoundThm}, we will write $A\lessapprox B$ to mean that $A\lesssim_{\eps}\delta^{-\eps}B$ for each $\eps>0$.  

\subsection{Thanks}
The author would like to thank Larry Guth and Nets Katz for many helpful conversations. The author would like to thank Jonathan Hickman for comments and corrections to a previous version of this manuscript.

\section{The polynomial Wolff axioms for nested sequences of varieties}\label{polyWolffNestedSequences}
In this section we will prove Theorem \ref{polyWolffFlagsVarities}. We will begin by recalling some standard tools from real algebraic geometry.

\subsection{Tools from real algebraic geometry}
We will begin with some basic definitions and results from real algebraic geometry. Further details can be found in standard references such as \cite[Chapter 2]{BCR} and in \cite[Section 3]{Burguet}.

If $S\subset\RR^n$ is a semi-algebraic set, we define the dimension of $S$ to be the Krull dimension of the ring $\RR[x_1,\ldots,x_n]/I(S)$, where $I(S)$ is the ideal of polynomials that vanish on $S$. Equivalently, the dimension of $S$ is the largest integer $d$ so that there exists a continuous injection $\phi\colon (0,1)^d\to S$.  In practice, we will be interested in two types of semi-algebraic sets. The first are semi-algebraic subsets of $\RR^n$ that have non-empty interior; such sets always have dimension $n$. The second are sets of the form $Z\cap B$, where $Z\subset\RR^n$ is an algebraic variety of dimension $d<n$ and $B$ is an open (Euclidean) ball; sets of this form always have dimension at most $d$. The following theorem of Wongkew \cite{Wongkew} bounds the covering number of the second type of semi-algebraic set.
\begin{thm}\label{WonkewThm}
Let $Z\subset\RR^n$ be a real algebraic variety of dimension $d$ whose defining polynomials have degree at most $E$. Let $B\subset\RR^n$ be a ball of radius $r$. Then there exists a constant $C(n)$ depending only on $n$ so that for all $\rho>0$,
$$
|N_{\rho}(Z\cap B)|\leq C(n)\sum_{j=n-d}^n E^j\rho^jr^{n-j}.
$$
In particular, if $0<\rho\leq r$ then there is a constant $C(n,E)$ so that
\begin{equation}\label{volumeBoundXcapB}
|N_{\rho}(Z\cap B)|\leq C(n,E) \rho^{n-d}r^{d}.
\end{equation}
\end{thm}

The following theorem of Milnor and Thom \cite{Mil} controls the number of (Euclidean) connected components of a real algebraic variety. While this theorem has seen numerous refinement and improvements, the original version is sufficient for our purposes.
\begin{thm}\label{MilnorThomTheorem}
Let $Z\subset\RR^n$ be a real algebraic variety of dimension $d$ whose defining polynomials have degree at most $E$. Then $Z$ has at most $E(2E-1)^{n-1}$ connected components. 
\end{thm}

We will also need to control the number of (Euclidean) connected components of a semi-algebraic set. The following special case of a result of Basu, Pollack, and Roy \cite{BPR} is sufficient for our needs. 
\begin{thm}\label{BPRTheorem}
Let $S\subset\RR^n$ be a semi-algebraic set of complexity $E$. Then there exists a constant $C(n,E)$ depending only on $n$ and $E$ so that $S$ has at most $C(n,E)$ connected components.
\end{thm}

One of the deepest results about semi-algebraic sets is the Yomdin-Gromov algebraic lemma. The version stated here is Theorem 1 from \cite{Burguet}.
\begin{thm}\label{gromovAlgLemma}
Let $m,d,r$ and $E$ be integers. Let $S\subset [0,1]^m$ be a compact semi-algebraic set of dimension $d$ and complexity at most $E$. Then there exists an integer $N=N(m,E,r)$ and continuous maps $\phi_1,\ldots,\phi_N\colon [0,1]^d\to[0,1]^m$ so that 
$$
S = \bigcup_{i=1}^N \phi_i([0,1]^d),
$$
and for each index $i$ we have
$$
\norm{\phi_i}_r = \max_{\beta\colon |\beta|\leq r}\Vert \partial^\beta \phi_i\Vert_{L^\infty((0,1)^d)}\leq 1.
$$
\end{thm}

The final result we will need is a lemma that allows us to select one representative from each fiber of a projection map between semi-algebraic sets. A proof this lemma can be found in \cite[Lemma 2.6]{Z} (cf. \cite[Lemma 2.2]{KR}).
\begin{lem}[Selecting one point from each fiber]\label{selectingPointFiberLem}
Let $m,n$ and $E$ be integers. Then there is a constant $C(m,n,E)>0$ so that the following holds. Let $S\subset [0,1]^m$ be a semi-algebraic set of complexity at most $E$ and let $f\colon S\to\RR^n$ be a function whose graph is semi-algebraic of complexity at most $E$. Then there exists a semi-algebraic set $U\subset S$ of complexity at most $C(m,n,E)$ so that $f(U) = f(S)$, and the restriction of $f$ to $U$ is an injection. 
\end{lem}

\subsection{Extending tubes inside semi-algebraic sets}
In this section, we will show that the set of tubes contained inside a semi-algebraic set cannot ``expand'' too much if we extend the tubes. If $T$ is a $\lambda\times\delta$ tube and if $A\geq 1$, we define $\operatorname{Ext}_A(T)$ to be the $A\lambda\times\delta$ tube that has the same midpoint and coaxial line as $T$. 

Recall that Besicovitch \cite{Bes} constructed a set $K\subset [0,2]^2$ of measure $\leq c(\delta)$ that contains a $1\times \delta$ tube pointing in every $\delta$-separated direction. The function $c(\delta)\to 0$ as $\delta\to 0.$ However, if each of these tubes are replaced by $\operatorname{Ext}_3(T)$, then the union of these extended tubes has volume $\sim 1$. The next lemma says that this type of phenomena is not possible if the set $K$ is semi-algebraic of bounded complexity.

\begin{lem}[Extending tubes inside semi-algebraic sets]\label{extSemiAlgSetsLem}
Let $n$ and $E$ be integers with $n\geq 2$, and let $\eps>0$. Then there is a constant $C(n,E,\eps)>0$ so that the following holds. Let $S\subset [0,1]^n$ be a semi-algebraic set of complexity at most $E$. Let $0<\delta\leq\lambda\leq 1$ and let $1\leq A\leq \lambda^{-1}$. Then 
\begin{equation}\label{volumeBoundDilS}
\Big|\bigcup_{\substack{T\ \textrm{a}\ \lambda\times\delta\ \textrm{tube}\\ T\subset S}}\operatorname{Ext}_A(T)\Big|\leq C(n,E,\eps)\delta^{-\eps}A^n |S|.
\end{equation}
\end{lem}

Lemma \ref{extSemiAlgSetsLem} will be proved by combining the Yomdin-Gromov algebraic lemma with the following elementary estimate, which says that if a univariate polynomial is small (on average) on an interval, then it cannot grow too quickly outside that interval. We will apply this lemma to a polynomial that measures the ``compression'' of tubes inside a semi-algebraic set---if the tubes are very compressed inside the set, then they must remain at least somewhat compressed when they are extended beyond the set. 
\begin{lem}
Let $P(x)$ be a polynomial of degree $\leq D$ and let $J\subset I\subset\RR$ be closed intervals. Then
\begin{equation}\label{supBoundedByIntegral}
\norm{P}_{L^\infty(I)} \leq C(D)  \Big(\frac{|I|}{|J|}\Big)^D \norm{P}_{\L{}^1(J)}.
\end{equation}
\end{lem}
\begin{proof}
Replacing $P(x)$ by $P(x-x_0)$ if necessary, we may assume that $J$ is centered at $0$. Let $J^\prime\subset J$ be a measurable set with $|J^\prime|\geq\frac{1}{2}|J|$ so that $|P(x)|\leq 2\norm{P}_{\L{}^1(J)}$ for all $x\in J^\prime$. Cover $J$ by intervals of length $\frac{1}{4(D+1)}$. Observe that at least $2(D+1)$ of these intervals must intersect $J^\prime$. Numbering these intervals from left to right and selecting one point from each interval with odd index, we conclude that there exist points $x_1,\ldots,x_{D+1}\in J^\prime$ so that $|x_i-x_j| \geq |J|/(4D+4)$ whenever $i\neq j$.

By Lagrange interpolation we can write
\begin{equation}\label{lagrangePoly}
P(x) = \sum_{j=1}^{D+1}P_j(x),\quad\textrm{where}\ P_j(x)= P(x_j)\prod_{\substack{k=1\\k\neq j}}^{D+1} \frac{x-x_k}{x_j-x_k}.
\end{equation}

For each index $j$, we have
$$
\Big|P(x_j)\prod_{\substack{k=1\\k\neq j}}^{D+1} \frac{x-x_k}{x_j-x_k}\Big|\leq \Big(2\norm{P}_{\L{}^1(J)}\Big)\Big(\frac{4D+4}{|J|}\Big)^D \big(|x|+|J|\Big)^D\lesssim_D \norm{P}_{\L{}^1(J)}\Big(1+\Big(\frac{|x|}{|J|}\Big)^D\Big),
$$
and thus
\begin{equation}\label{PxBound}
|P(x)|\lesssim_D \norm{P}_{\L{}^1(J)} \Big(1+\Big(\frac{|x|}{|J|}\Big)^D\Big).
\end{equation}
Since $J$ is centered at $0$ and $J\subset I$, we have that $|x|\leq |I|$ for all $x\in I$, and \eqref{supBoundedByIntegral} now follows from \eqref{PxBound}.
\end{proof}

We are now ready to prove Lemma \ref{extSemiAlgSetsLem}.

\begin{proof}[Proof of Lemma \ref{extSemiAlgSetsLem}]
Our proof will use many of the ideas developed by Katz and Rogers in \cite{KR}, and parts of the proof will closely mirror their arguments. We will begin with a few reductions.
\medskip

\noindent{\bf Reduction 1: $S$ has small diameter}\\
Suppose for the moment that there is a constant $C_1(n,E,\eps)>0$ so that for all $0<\delta\leq\lambda\leq A^{-1}\leq 1$, all semi-algebraic sets $S\subset [0,1/2]^n$ of diameter at most $\frac{3}{2}\lambda$ and complexity at most $E$, and all $\eps>0$, 
\begin{equation}\label{volumeBoundDilSSmallDiam}
\Big|\bigcup_{\substack{T\ \textrm{a}\ \lambda\times\delta\ \textrm{tube}\\T\subset S}}\operatorname{Ext}_A(T)\Big|\leq C_1(n,E,\eps)\delta^{-\eps}A^n |S|.
\end{equation}
With this assumption, let $0<\delta\leq\lambda\leq 1$, let $\eps>0$, and let $S\subset [0,1]^n$ be a semi-algebraic set of complexity at most $E$. We will show that there exists a constant $C(n,E,\eps)$ so that \eqref{volumeBoundDilS} holds. 

Let $\mathcal{B}$ be a set of balls of diameter $\frac{3}{2}\lambda$ with the property that each point in $[0,1]^n$ is contained in $O_n(1)$ balls from $\mathcal{B}$, and each ball of diameter $\frac{4}{3}\lambda$ is entirely contained in one of the balls from $\mathcal{B}$. Then since each $\lambda\times\delta$ tube is contained in a ball of diameter $\frac{4}{3}\lambda$, we have
$$
\{T\ \textrm{a}\ \lambda\times\delta\ \textrm{tube}, T\subset S\} =\bigcup_{B\in\mathcal{B}} \{T\ \textrm{a}\ \lambda\times\delta\ \textrm{tube}, T\subset S\cap B\}.
$$
Note that $S\cap B$ is also semi-algebraic, and the complexity of $S\cap B$ is bounded by a number that depends only on $n$ and the complexity of $S$. Applying \eqref{volumeBoundDilSSmallDiam}, we have
 \begin{equation}
 \begin{split}
 \Big|\bigcup_{\substack{T\ \textrm{a}\ \lambda\times\delta\ \textrm{tube}\\ T\subset S}}\operatorname{Ext}_A(T)\Big|
 &= \Big|\bigcup_{B\in\mathcal{B}} \bigcup_{\substack{T\ \textrm{a}\ \lambda\times\delta\ \textrm{tube}\\ T\subset S\cap B}}\operatorname{Ext}_A(T)\Big|\\
 &\leq \sum_{B\in\mathcal{B}}C_1 \delta^{-\eps}A^n|S\cap B|\\
 &\lesssim C_1(n,E,\eps) \delta^{-\eps}A^n|S|.
 \end{split}
 \end{equation}
 Thus if $C(n,E,\eps)$ is selected sufficiently large (depending only on $n,\eps$ and $C_1(n,E,\eps)$, which in turn depends only on $n$, $\eps$, and $E$), then \eqref{volumeBoundDilS} holds.

\medskip

\noindent{\bf Reduction 2: All tubes point in almost the same direction}\\
Suppose for the moment that there is a constant $C_2(n,E,\eps)>0$ so that for all $0<\delta\leq\lambda\leq A^{-1}\leq 1$, all semi-algebraic sets $S\subset [0,1]^n$ of diameter at most $\frac{3}{2}\lambda$ and complexity at most $E$, and all $\eps>0$,
\begin{equation}\label{volumeBoundDilSSameDiretion}
\Big|\bigcup_{\substack{T\ \textrm{a}\ \lambda\times\delta\ \textrm{tube}\\ \angle(v(T),e_n)\leq\frac{1}{10}\\T\subset S}}\operatorname{Ext}_A(T)\Big|\leq C_2(n,E,\eps)\delta^{-\eps}A^n |S|.
\end{equation}
With this assumption, let $0<\delta\leq\lambda\leq 1$, let $\eps>0$, and let $S\subset [0,1]^n$ be a semi-algebraic set of diameter at most $\frac{3}{2}\lambda$ and complexity at most $E$. We will show that there exists a constant $C_1(n,E,\eps)$ so that \eqref{volumeBoundDilSSmallDiam} holds. 

Let $\Omega\subset S^{n-1}$ be a set of $O_n(1)$ unit vectors so that each unit vector in $S^{n-1}$ makes an angle $\leq 1/10$ with a vector from $\Omega$. For each $v\in\Omega$, let $\mathcal{O}_v$ be an orthogonal transformation taking $v$ to the $n$-th basis vector $e_n$ and let $S_v = \mathcal{O}_v(S)$. Then
$$
\bigcup_{\substack{T\ \textrm{a}\ \lambda\times\delta\ \textrm{tube}\\ T\subset S}}\operatorname{Ext}_A(T)=\bigcup_{v\in \Omega}\mathcal{O}_v^{-1}\Big(\bigcup_{\substack{T\ \textrm{a}\ \lambda\times\delta\ \textrm{tube}\\ \angle(v(T),e_n)\leq\frac{1}{10}\\T\subset S_v}}\operatorname{Ext}_A(T)\Big).
$$
Applying \eqref{volumeBoundDilSSameDiretion} to each set $S_v$, we conclude that
\begin{equation}
\begin{split}
\Big|\bigcup_{\substack{T\ \textrm{a}\ \lambda\times\delta\ \textrm{tube}\\ T\subset S}}\operatorname{Ext}_A(T)\Big|
&\leq \sum_{v\in\Omega}\Big|\bigcup_{\substack{T\ \textrm{a}\ \lambda\times\delta\ \textrm{tube}\\ \angle(v(T),e_n)\leq\frac{1}{10}\\T\subset S_v}}\operatorname{Ext}_A(T)\Big|\\
&\leq \sum_{v\in\Omega} C_2(n,E,\eps)\delta^{-\eps}A^n|S_v|\\
&\lesssim C_2\delta^{-\eps}A^n|S|.
\end{split}
\end{equation}
Thus if $C_1(n,E,\eps)$ is selected sufficiently large (depending only on $n,\eps$ and $C_2(n,E,\eps)$, which in turn depends only on $n$, $\eps$, and $E$), then \eqref{volumeBoundDilSSmallDiam} holds. 

\medskip
\noindent{\bf The main argument}
Let $0<\delta\leq\lambda\leq 1$, let $\eps>0,$ and let $S\subset B(0,\frac{3}{4}\lambda)\subset [0,1]^n$ be a semi-algebraic set of complexity at most $E$. We need to show that there exists a constant $C_2(n,E,\eps)$ so that \eqref{volumeBoundDilSSameDiretion} holds.

It will be convenient to replace $S$ with the set
$$
S^\prime = \bigcup_{\substack{T\ \textrm{a}\ \lambda\times\delta\ \textrm{tube}\\ \angle(v(T),e_n)\leq\frac{1}{10}\\T\subset S}}T.
$$
Note that $|S^\prime|\leq |S|$, and 
$$
\bigcup_{\substack{T\ \textrm{a}\ \lambda\times\delta\ \textrm{tube}\\ \angle(v(T),e_n)\leq\frac{1}{10}\\T\subset S^\prime}}\operatorname{Ext}_A(T)=\bigcup_{\substack{T\ \textrm{a}\ \lambda\times\delta\ \textrm{tube}\\ \angle(v(T),e_n)\leq\frac{1}{10}\\T\subset S}}\operatorname{Ext}_A(T).
$$

Since $\lambda\geq\delta$, we have that either $S^\prime$ is empty or $|S^\prime|\gtrsim\lambda\delta^{n-1}\geq\delta^n$. If the former holds then Lemma \ref{extSemiAlgSetsLem} is trivially true, so we can assume that
\begin{equation}\label{trivialVolumeBoundSPrime}
|S^\prime|\gtrsim\delta^n.
\end{equation}

Observe that there exists a constant $c_n>0$ so that whenever $T$ is a $\lambda\times\delta$ tube and whenever $x\in T$, then $|B(x,\delta)\cap T|\geq c_n|B(x,\delta)|$. Since $S^\prime$ is a union of $\lambda\times\delta$ tubes, this implies that for all $x\in S^\prime$, $|B(x,\delta)\cap S^\prime|\geq c_n|S^\prime|$, and thus
\begin{equation}\label{coveringNumberSPrime}
\mathcal{E}_\delta(S^\prime)\leq c_n^{-1} \delta^{-n}|S^\prime|.
\end{equation}
In particular, this means that for all $\rho\geq\delta$, we have
\begin{equation}\label{neighborhoodVolumeBound}
|N_{\rho}(S^\prime)|\lesssim (\rho/\delta)^n|S^\prime|\leq (\rho/\delta)^n|S|,
\end{equation}
where the implicit constant depends only on $n$. Note that inequality \eqref{neighborhoodVolumeBound} might not be true if the left hand side was replaced by $|N_{\rho}(S)|$, which is why the set $S^\prime$ was introduced.

For each $t\in\RR$, define the ``vertical'' hyperplane 
$$
H_t = \{(x_1,\ldots,x_n)\in \RR^n\colon x_n=t\}.
$$
Since $S^\prime\subset B(0,\frac{3}{4}\lambda)$, for each $\lambda\times\delta$ tube $T$ with $T\subset S^\prime$ and $\angle(v(T),e_n)\leq\frac{1}{10}$, we have that every line segment of length $\lambda$ that is contained in $T$ and parallel to $v(T)$ intersects $H_0$ and $H_{\lambda/8}$. Since $\angle(w,e_n)\leq 1/10$ (and in particular, $w$ is not perpendicular to $e_n$), these intersection points are unique.

Define
\begin{equation}\label{defnOfCalL}
\begin{split}
\mathcal{L}=\{(a,d)\in \RR^{2n-2}\colon &\textrm{there exists a}\ \lambda\times\delta\ \textrm{tube}\ T\subset S^\prime\ \textrm{with}\ \angle(v(T),e_n)\leq 1/10,\\
&\textrm{so that}\ v(T)\ \textrm{is parallel to}\ (d,1),\ \textrm{and}\ (a,0)\in H_0\cap T\}.
\end{split}
\end{equation}
Since $S^\prime\subset [0,1]^n$, we have $\mathcal{L}\subset[0,1]^{2n-2}$. The key observations are that 
\begin{equation}\label{asdContained}
\bigcup_{(a,d)\in\mathcal{L}}\{(a,0)+s(d,1)\colon s\in[0,\lambda/8]\}\subset S^\prime,
\end{equation}
and
\begin{equation}\label{tildeSBoundsExt}
 \bigcup_{\substack{T\ \textrm{a}\ \lambda\times\delta\ \textrm{tube}\\ \angle(v(T),e_n)\leq\frac{1}{10}\\T\subset S}}\operatorname{Ext}_A(T)\ \subset\ \bigcup_{(a,d)\in\mathcal{L}}\{(a,0)+s(d,1)\colon s\in[0,2A\lambda]\}.
\end{equation}
The containment \eqref{asdContained} follows from the definition \eqref{defnOfCalL} of $\mathcal{L}$. To verify \eqref{tildeSBoundsExt}, let $x$ be a point in the left hand side of \eqref{tildeSBoundsExt}. Then there is a $\lambda\times\delta$ tube $T$ with $T\subset S,$ $\angle(v(T),e_n)\leq 1/10$, and $x\in \operatorname{Ext}_A(T)$. Let $L=\{(a,0)+s(d,1)\colon s\in\RR\}$ be the line that contains $x$ and is parallel to $v(T)$. Since $L$ intersects $\operatorname{Ext}_A(T)$ in at least one point (i.e. at the point $x$), and $L$ is parallel to $v(T)$, we have that $L\cap H_0 \subset T\cap H_0$, and thus $(a,b)\in\mathcal{L}$. This implies that $x$ is in the right hand side of \eqref{tildeSBoundsExt}.

Define
$$
\tilde S = \bigcup_{(a,d)\in\mathcal{L}}\{(a,0)+s(d,1)\colon s\in[0,2A\lambda]\}.
$$
Then
\begin{equation}\label{tildeSContainedBall}
\bigcup_{\substack{T\ \textrm{a}\ \lambda\times\delta\ \textrm{tube}\\ \angle(v(T),e_n)\leq\frac{1}{10}\\T\subset S}}\operatorname{Ext}_A(T)\ \subset\  \tilde S\ \subset\ [-2A\lambda,2A\lambda]^n.
\end{equation}
We will prove that there exists a constant $C_3(n,E,\eps)$ so that for each $t\in [-2A\lambda,\ 2A\lambda]$, 
\begin{equation}\label{boundingVolumeOfASlice}
| \tilde S \cap H_t| \leq C_3(n,E,\eps)\delta^{-\eps}\big(\lambda^{-1} A^{n-1}|N_{\delta}(S^\prime)|+\delta^n\big),
\end{equation}
where the $|\cdot|$ on the left denotes $(n-1)$ dimensional Lebesgue measure and the $|\cdot|$ on the right denotes $n$ dimensional Lebesgue measure.

Assuming that \eqref{boundingVolumeOfASlice} is true for the moment, we can integrate in $t$ to conclude that
\begin{equation}\label{boundOnTildeSSlice}
\begin{split}
|\tilde S| & = \int_{-2 A\lambda}^{2 A\lambda}| \tilde S \cap H_t|dt\\
&\leq(4 A\lambda)C_3(n,E,\eps)\delta^{-\eps}\big(\lambda^{-1} A^{n-1}|N_{\delta}(S^\prime)|+\delta^n\big)\\
& \lesssim C_3(n,E,\eps)\delta^{-\eps} A^n|N_{\delta}(S^\prime)|,
\end{split}
\end{equation}
where on the last line we used \eqref{trivialVolumeBoundSPrime} and the fact that $\lambda^{-1} A^{n-1}\geq 1$. Combining \eqref{neighborhoodVolumeBound}, \eqref{tildeSContainedBall}, and \eqref{boundOnTildeSSlice} would establish \eqref{volumeBoundDilSSameDiretion}. The remainder of the argument will be devoted to proving \eqref{boundingVolumeOfASlice}.

Fix a number $t\in [-2 A\lambda,\ 2A\lambda]$. We will prove that \eqref{boundingVolumeOfASlice} is true for this choice of $t$. First, we can assume that $\tilde S\cap H_t$ has dimension $n-1$, since otherwise $|\tilde S\cap H_t|=0$ and the result is immediate.

Consider the map $\phi\colon \mathcal{L}\to \tilde S \cap H_t$ given by $(a,d)\mapsto (a,0)+t(d,1)$. Use Lemma \ref{selectingPointFiberLem} to select a semi-algebraic set $\mathcal{L}^\prime\subset\mathcal{L}$ of dimension $n-1$ and complexity at most $C(n,E)$ so that the restriction of $\phi$ to $\mathcal{L}^\prime$ is a bijection.

Apply Theorem \ref{gromovAlgLemma} with $m=d=n-1$ and $r=2(n-1)^2/\eps$. There is a number $N\lesssim 1$ and maps 
$$
(F_i,G_i)\colon[0,1]^{n-1}\to[0,1]^{2n-2},\quad i=1,\ldots,N,
$$ 
so that
\begin{equation}\label{LPrimeCovered}
\mathcal{L}^\prime = \bigcup_{i=1}^N (F_i,G_i)([0,1]^{n-1}).
\end{equation}
and
\begin{equation}\label{derivativeBoundFromGromov}
\Vert (F_i,G_i)\Vert_r \leq 1\quad\textrm{for each index}\ i.
\end{equation}

Since
$$
\tilde S\cap H_t = \{(a,0)+t(d,1) \colon (a,d)\in\mathcal{L}\},
$$ 
\eqref{LPrimeCovered} implies that
$$
\tilde S\cap H_t =\bigcup_{i=1}^N \{(a,0)+t(d,1) \colon (a,d)\in (F_i,G_i)([0,1]^{n-1}) \}. 
$$
Thus by pigeonholing there exists an index $i_0$ so that
$$
 |\{(a,0)+t(d,1) \colon (a,d)\in (F_{i_0},G_{i_0})([0,1]^{n-1}) \}|\gtrsim |\tilde S \cap H_t|.
$$

Since $[0,1]^{n-1}$ can be covered by $\lesssim\delta^{-\eps}$ balls of radius $\delta^{\eps/(n-1)}$, there exists a point $x_0\in [0,1]^{n-1}$ so that if we define $U=[0,1]^{n-1}\cap B(x_0,\delta^{\frac{\eps}{n-1}})$, then 
$$
|\{(a,0)+t(d,1) \colon (a,d)\in (F_{i_0},G_{i_0})(U) \}|\gtrsim \delta^{\eps}|\tilde S \cap H_t|.
$$

Let $F$ (resp. $G$) be the degree $r-1$ polynomial given by the $(r-1)$-st order Taylor expansion of $F_{i_0}$ (resp. $G_{i_0}$) around $x_0$. By \eqref{derivativeBoundFromGromov}, we have
\begin{equation}\label{FGCloseToTaylorApprox}
|(F+G)(x)-(F_{i_0}+G_{i_0})(x)|\leq |x-x_0|^r\leq  \delta^{2n-2}\quad\textrm{for all}\ x\in U.
\end{equation}
This implies that 
\begin{equation}\label{ContainedInNbhdSPrime}
\{(a,0)+s(d,1) \colon (a,d)\in (F,G)(U)\} \subset N_{\delta}(S^\prime)\ \textrm{for all}\ s\in [0,\lambda/8].
\end{equation}

We claim that there exists a constant $C_3(n)$ so that
\begin{equation}\label{mostOfTildeScapHCaptured}
|\{(a,0)+t(d,1) \colon (a,d)\in (F,G)(U) \}|\gtrsim \delta^{\eps}|\tilde S \cap H_t|-C_3(n)\delta^{n}.
\end{equation}
To see this, define 
\begin{equation*}
\begin{split}
J&=\operatorname{bdry}(\{(F(x),0)+t(G(x),1))\colon x\in U\}),\\
J^\prime&=\operatorname{bdry}(\{(F_{i_0}(x),0)+t(G_{i_0}(x),1))\colon x\in U\}).
\end{split}
\end{equation*}
Observe that since the maps $x\mapsto (F(x),0)+t(G(x),1))$ and $x\mapsto (F_{i_0}(x),0)+t(G_{i_0}(x),1))$ are continuous, they map $\operatorname{bdry}(U)$ to $J$ and $J^\prime$, respectively. By \eqref{FGCloseToTaylorApprox}, we have that $J\subset N_{\delta^{2n-2}}(J^\prime)$ and $J^\prime\subset N_{\delta^{2n-2}}(J)$. For each $\rho>0$, 
$$
\mathcal{E}_\rho(\operatorname{bdry}(U))\lesssim \rho^{2-n},
$$ 
where the implicit constant depends only on $n$. By \eqref{derivativeBoundFromGromov} we have that $(F_{i_0},G_{i_0})$ is 1-Lipschitz, and thus $\mathcal{E}_{\rho}(J^\prime)\lesssim \rho^{2-n}$ for all $\rho>0$. This implies that $|N_{\delta^{2n-2}}(J)|\lesssim \delta^n$. Thus if we select $C_3(n)$ sufficiently large, we have
\begin{equation}\label{lowerBoundFGU}
\begin{split}
|\{(a,0)+t(d,1) \colon (a,d)\in (F,G)(U) \}|&  \geq |\{(a,0)+t(d,1) \colon (a,d)\in (F_{i_0},G_{i_0})(U) \}|-|N_{\delta^{2n-2}}(J)|\\
& \gtrsim \delta^{\eps}|\tilde S \cap H_t| - C_3(n)\delta^n,
\end{split}
\end{equation}
which gives us \eqref{mostOfTildeScapHCaptured}.

At this point, \eqref{mostOfTildeScapHCaptured} gives us a lower bound on the size of the slice 
$$
\{(a,0)+t(d,1) \colon (a,d)\in (F,G)(U) \},
$$
and \eqref{ContainedInNbhdSPrime} gives us an upper bound on the size of each of the slices 
$$
\{(a,0)+s(d,1) \colon (a,d)\in (F,G)(U)\},\quad s\in [0,\lambda/8].
$$
Our next task is to compare these lower and upper bounds. To do this we will need to introduce the change of variables formula from multivariate calculus. The version stated here is Theorem 9.9.3 from \cite{Kuttler} (see also \cite[Theorem 20.15]{Kuttler2} for a similar formulation).
\begin{thm}\label{changeOfVariablesSpivak}
Let $U\subset \RR^m$ be an open set, let $h\colon U\to\RR^m$ be $C^1$, and let $V = h(U)$. For each $y\in V$, define $m(y)=\{x\in U\colon h(x)=y\};$ this is defined whenever $\{x\in U\colon h(x)=y\}$ is finite. 

Then $m$ is defined almost everywhere; $m$ is measurable (with respect to $m$-dimensional Lebesgue measure); and 
$$
\int_{V}m(y)dy = \int_U |\operatorname{det}(Dh(x))|dx,
$$
where both integrals are with respect to $m$-dimensional Lebesgue measure.
\end{thm}

Since $m(y)\geq 1$ for all $y\in h(U)$ except on a set of measure 0 (for which it is not defined), we immediately obtain the following corollary.
\begin{cor}\label{changeOfVariableIneq}
Let $U\subset \RR^m$ be an open set and let $h\colon U\to\RR^m$ be $C^1$. Then
$$
|h(U)|\leq \int_U |\operatorname{det}(Dh(x))|dx.
$$
\end{cor}
Applying Corollary \ref{changeOfVariableIneq} with $h(x) = F(x)+tG(x)$ and $U$ as above, we conclude that
\begin{equation}\label{lowerBoundFtGB}
|(F+tG)(U)|\leq \int_U |\operatorname{det}(DF(x) + tDG(x))|dx.
\end{equation}

On the other hand, by Theorem \ref{MilnorThomTheorem} we have that for each $s\in\RR$ and each $y\in (F+sG)(U)$, either the set $\{x \in U\colon (F+sG)(x)=y\}$ is infinite, or
$$
\#\{x \in U\colon (F+sG)(x)=y\}\leq r(2r-1)^{n-1}\leq (2r)^n.
$$
In particular, we have that for each $s\in\RR$, 
\begin{equation}\label{lowerBoundForFsGU}
|(F+sG)(U)|\geq (2r)^{-n}\int_U |\operatorname{det}(DF(x)+sDG(x))|dx.
\end{equation}


For each $x\in U$, write
$$
|\operatorname{det}(DF(x)+sDG(x))|=|P_x(s)|,
$$
where $P_x(s)$ is a polynomial of degree at most $n-1$. By Lemma \ref{supBoundedByIntegral}, we have that
\begin{equation}\label{boundOnPxT}
|P_x(t)| \lesssim \Big(\frac{A\lambda }{\lambda/8}\Big)^{n-1} \norm{P}_{\L{}^1([0,\lambda/8])}\lesssim \lambda^{-1}A^{n-1}\int_{0}^{\lambda/8}|P_x(s)|ds.
\end{equation}
We are now ready to prove \eqref{boundingVolumeOfASlice}. From \eqref{mostOfTildeScapHCaptured} we have
$$
\delta^{\eps}|\tilde S \cap H_t|-C_3(n)\delta^n \leq |(F+tG)(U)|.
$$
By \eqref{lowerBoundFtGB} and \eqref{boundOnPxT}, we have
\begin{equation*}
\begin{split}
|(F+tG)(U)| & \leq \int_U |\operatorname{det}(DF(x) + tDG(x))|dx\\
& = \int_U |P_x(t)|dx\\
& \lesssim \int_U\Big(\lambda^{-1}A^{n-1}\int_{0}^{\lambda/8}|P_x(s)|ds\Big)dx\\
& =\lambda^{-1}A^{n-1}\int_{0}^{\lambda/8}\int_U |DF(x)+sDG(x)| dsdx.
\end{split}
\end{equation*}
By \eqref{lowerBoundForFsGU},
$$
\int_{0}^{\lambda/8}\int_U |DF(x)+sDG(x)| dsdx \leq r^{2n} \int_{0}^{\lambda/8}|(F+sG)(U)|ds. 
$$
Finally, by \eqref{ContainedInNbhdSPrime} we have
$$
\int_{0}^{\lambda/8}|(F+sG)(U)|ds\leq |N_{\delta}(S^\prime)|.
$$
Combining these inequalities we conclude that
$$
\delta^{\eps}|\tilde S \cap H_t|-C_3(n)\delta^n\lesssim r^{2n}\lambda^{-1}A^{n-1}|N_{\delta}(S^\prime)|.
$$
Since $r=2(n-1)^2/\eps$ depends only on $\eps$ and $n$, this establishes \eqref{boundingVolumeOfASlice} and completes the proof.
\end{proof}

We will finish this section with a corollary of Lemma \ref{extSemiAlgSetsLem} that allows us to replace the extension $\operatorname{Ext}_A(T)$ of the tube $T$ with a slightly more useful ``fattening'' of $T$. If $T$ is a $\lambda\times\delta$ tube and if $A\geq 1$, we define $\operatorname{Fat}_A(T)$ to be the $A\lambda\times A\delta$ tube that has the same midpoint and coaxial line as $T$.
\begin{cor}[Fattening tubes inside semi-algebraic sets]\label{FatSemiAlgSetsLem}
Let $n$ and $E$ be integers with $n\geq 2$, and let $\eps>0$. Then there is a constant $C(n,E,\eps)>0$ so that the following holds. Let $S\subset [0,1]^n$ be a semi-algebraic set of complexity at most $E$. Let $0<\delta\leq\lambda\leq 1$ and let $1\leq A\leq \lambda^{-1}$. Then 
\begin{equation}\label{volumeBoundDilS}
\Big|\bigcup_{\substack{T\ \textrm{a}\ \lambda\times\delta\ \textrm{tube}\\ T\subset S}}\operatorname{Fat}_A(T)\Big|\leq C(n,E,\eps)\delta^{-\eps}A^n |S|.
\end{equation}
\end{cor}
\begin{proof}
Define
\begin{equation*}
\begin{split}
S^\prime &= \bigcup_{\substack{T\ \textrm{a}\ \lambda/2 \times\delta/2\ \textrm{tube}\\ T\subset S}}\operatorname{Ext}_{2A}(T),\\
S^{\prime\prime} & = \bigcup_{\substack{T\ \textrm{a}\ A\lambda/2 \times A\delta/2\ \textrm{tube}\\ T\subset S^\prime}}\operatorname{Ext}_4(T).
\end{split}
\end{equation*}
By Lemma \ref{extSemiAlgSetsLem}, there is a constant $C$ (depending on $n,E$ and $\eps$) so that
$$
|S^{\prime\prime}|\leq C\delta^{-\eps/2} 8^n |S^\prime| \leq C^2\delta^{-\eps}(8A)^n |S|.
$$
All that remains is to verify that
$$
\bigcup_{\substack{T\ \textrm{a}\ \lambda\times\delta\ \textrm{tube}\\ T\subset S}}\operatorname{Fat}_A(T)\subset S^{\prime\prime}.
$$
To see this, let $T_0\subset S$ be a $\lambda\times\delta$ tube. Then
$$
\bigcup_{\substack{T\ \textrm{a}\ \lambda/2 \times\delta/2\ \textrm{tube}\\ T\subset T_0}}\operatorname{Ext}_{2A}(T)
$$
contains a $A\lambda/2 \times A\delta$ tube with the same coaxial line as $T_0$. The midpoint of this tube has distance $\leq A\lambda$ from the midpoint of $T_0$. Call this tube $T_1$. Finally,
$$
\operatorname{Fat}(T_0)\subset  \bigcup_{\substack{T\ \textrm{a}\ A\lambda/2 \times A\delta/2\ \textrm{tube}\\ T\subset T_1}}\operatorname{Ext}_4(T).\qedhere
$$
\end{proof}

\subsection{Tubes inside semi-algebraic sets}
\begin{lem}\label{CountingTubesInsideSetsLem}
Let $n,E,$ and $K$ be integers with $n\geq 2$, and let $\eps>0$. Then there is a constant $C(n,E,K,\eps)>0$ so that the following holds. Let $S_1\supset \ldots\supset S_d$ be semi-algebraic sets of complexity at most $E$. Let $r_1\geq r_2\geq\ldots\geq r_d>0$ and $\rho_1\geq\rho_2\geq\ldots\geq\rho_d>0$.
Suppose that for each index $i$, $S_i$ has diameter $r_i$ and obeys the growth condition
\begin{equation}\label{SiGrowthCondition}
|N_{\rho_i}(S_i)\cap B(x,r)|\leq E\rho_i^{i}r^{n-i}\quad\textrm{for all balls}\ B(x,r).
\end{equation}
Let $0<\delta\leq\rho_1/r_1$, and let $\lines$ be a set of lines pointing in $\delta$-separated directions with the property that for each $L\in\lines$ and each index $i$,
\begin{equation}\label{TIntersectsEachVariety}
L\cap N_{\rho_i}(S_i) \ \textrm{contains a line segment of length}\ r_i/K.
\end{equation}
Then
\begin{equation}\label{boundNumberDirectionSeparatedTubes}
\#\lines \leq C(n,E,K,\eps)\big(\frac{r_1}{\rho_d}\big)^{\eps} \delta^{1-n-\eps}  \frac{\rho_1\cdots\rho_d}{r_1\cdots r_d}.
\end{equation}
\end{lem}

\begin{proof}
By translating if necessary, we can assume that $0\in S_d$, and thus $S_i\subset B(0,r_i)$ for each index $i$. Next, observe that Lemma \ref{CountingTubesInsideSetsLem} is dilation invariant. Thus we may assume that $r_1 = 1$, and in particular each set $S_i$ is contained in $[0,1]^n$ (even though $r_1=1$, in the arguments below we will often keep track of terms involving $r_1$, as we believe this adds clarity to the arguments). Finally, we may assume that $r_i> \rho_i$ for each index $i$; indeed if this inequality fails for some index $i$, then we may simply omit the corresponding requirement \eqref{TIntersectsEachVariety} and re-index the remaining sets (note that the growth condition \eqref{SiGrowthCondition} remains true after this re-indexing).

Define $\tilde S_d=N_{2\rho_d}(S_d)$. For each $j=1,\ldots,d-1,$ define
$$
\tilde S_{d-j}=N_{2\rho_{d-j}}(S_{d-j}) \cap \bigcup_{\substack{T\ \textrm{a}\ r_{d-j+1}/K\times\rho_{d-j+1}\ \textrm{tube}\\ T\subset \tilde S_{d-j+1}}}\operatorname{Fat}_{\frac{Kr_{d-j}}{r_{d-j+1}}}(T).
$$
It is easy to verify that $\tilde S_{d-j}$ is semi-algebraic, and the complexity of $\tilde S_{d-j}$ depends only on $E,n,$ and $j$.

Let $L\in\lines$. Since \eqref{TIntersectsEachVariety} is true for $i=d$, we have that $L\cap N_{\rho_d}$ contains a line segment of length $r_d/K$, and thus there is a $r_d/K\times\rho_d$ tube contained in $\tilde S_d$ whose coaxial line is $L$; this tube must be contained in $B(0,r_d)$. We conclude that
$$
L\cap B(0,r_{d-1})\subset \bigcup_{\substack{T\ \textrm{a}\ r_{d}/K\times\rho_{d}\ \textrm{tube}\\ T\subset \tilde S_{d}}}\operatorname{Fat}_{\frac{Kr_{d-1}}{r_{d}}}(T).
$$
Since \eqref{TIntersectsEachVariety} is true for $i=d-1$, we also have that $L\cap N_{\rho_{d-1}}(S_{d-1})=L\cap N_{\rho_{d-1}}(S_{d-1})\cap B(0,r_{d-1})$ contains a line segment of length $r_{d-1}/K$, and thus there is a $r_{d-1}/K\times\rho_{d-1}$ tube contained in $\tilde S_{d-1}$ whose coaxial line is $L$; this tube must be contained in $B(0,r_{d-1})$. We conclude that 
$$
L\cap B(0,r_{d-2})\subset \bigcup_{\substack{T\ \textrm{a}\ r_{d-1}/K\times\rho_{d-1}\ \textrm{tube}\\ T\subset \tilde S_{d-1}}}\operatorname{Fat}_{\frac{Kr_{d-2}}{r_{d-1}}}(T).
$$
An identical argument to the one above shows that $L\cap \tilde S_{d-1}$ contains a line segment of length $r_{d-2}/K$, and this line segment must be contained in $B(0,r_{d-2})$. Iterating this argument, we conclude that $L\cap \tilde S_1$ contains a line segment of length $r_1/K$, and thus there is a $r_1/K\times\rho_1$ tube contained in $\tilde S_1$ whose coaxial line is $L$. Since $\delta\leq\rho_1$, there is also a $r_1/K\times \delta$ tube contained in $\tilde S_1$ whose coaxial line is $L$. The lines in $L$ point in $\delta$-separated directions, so we can apply Theorem \ref{polyWolffAxiomsThm} to conclude that
\begin{equation}\label{boundOnTubesSzS1}
\#\lines\leq C^\prime(n,E,\eps)|\tilde S_1|\delta^{1-n-\eps}r_1^{-n}K^n.
\end{equation}
Our next task is to bound $|\tilde S_1|$.

We will prove by induction that for each $j=0,1,\ldots,d-1$,
\begin{equation}\label{ineqForSdj}
|\tilde S_{d-j}| \leq C_j(n,E,K,\eps)\rho_d^{-\frac{j\eps}{d-1}}\rho_{d-j}^{d-j} r_{d-j}^{n-d+j}\prod_{i=0}^{j-1} (\rho_{d-i}r_{d-i}^{-1}).
\end{equation}
When $j=0$, \eqref{ineqForSdj} is the bound 
$$
|\tilde S_d|\leq C_0(n,E,K,\eps)\rho_d^{d}r_{d}^{n-d},
$$ 
which follows from the growth condition \eqref{SiGrowthCondition}. 

Suppose now that the result has been proved for some $j<d-1$. Let
$$
W_{d-(j+1)} = \bigcup_{\substack{T\ \textrm{a}\ r_{d-j}/K\times\rho_{d-j}\ \textrm{tube}\\ T\subset \tilde S_{d-j}}}\operatorname{Fat}_{\frac{Kr_{d-(j+1)}}{r_{d-j}}}(T).
$$
By Corollary \ref{FatSemiAlgSetsLem}, we have
\begin{equation}\label{volumeFattenedTubes}
\begin{split}
|W_{d-(j+1)}| & \lesssim \rho_{d-j}^{-\frac{\eps}{d-1}} \Big(\frac{Kr_{d-(j+1)}}{r_{d-j}}\Big)^n|\tilde S_{d-j}|,\\
& \lesssim \rho_{d-j}^{-\frac{\eps}{d-1}} \Big(\frac{r_{d-(j+1)}}{r_{d-j}}\Big)^n|\tilde S_{d-j}|,
\end{split}
\end{equation}
where here (and throughout this argument) the implicit constant depends on $n,E,K,$ and $\eps$. Note that $W_{d-(j+1)}$ is a union of tubes of thickness $\rho_{d-j} \frac{Kr_{d-(j+1)}}{r_{d-j}}$, and thus if we select $r\sim  \rho_{d-j} \frac{Kr_{d-(j+1)}}{r_{d-j}}$, then

\begin{equation}
\begin{split}
\mathcal{E}_{r}(W_{d-(j+1)})
&\lesssim r^{-n} |W_{d-(j+1)}|\\
& \lesssim \rho_{d-j}^{-n -\frac{\eps}{d-1}} |\tilde S_{d-j}|.
\end{split}
\end{equation}
Applying the growth condition \eqref{SiGrowthCondition} with this value of $r$, we conclude that
\begin{equation}
\begin{split}
|\tilde S_{d-(j+1)}| & = |N_{2\rho_{d-(j+1)}}(S_{d-(j+1)}) \cap W_{d-(j+1)}| \\
&\lesssim \mathcal{E}_{r}(W_{d-(j+1)})\big(E\rho_{d-(j+1)}^{d-(j+1)}r^{n-d+(j+1)}\big)\\
&\lesssim \Big(\rho_{d-j}^{-n -\frac{\eps}{d-1}} |\tilde S_{d-j}|\Big) \Big(\rho_{d-(j+1)}^{d-(j+1)}\big(\rho_{d-j} \frac{r_{d-(j+1)}}{r_{d-j}}\big)^{n-d+(j+1)}\Big)\\
&\leq \rho_{d}^{ -\frac{\eps}{d-1}}\Big(\frac{\rho_{d-(j+1)}}{\rho_{d-j}}\Big)^{d-(j+1)}\Big(\frac{r_{d-(j+1)}}{r_{d-j}}\Big)^{n-d+(j+1)}|\tilde S_{d-j}|.
\end{split}
\end{equation}
Applying the induction hypothesis, we obtain

\begin{equation}\label{inductionClose}
\begin{split}
|\tilde S_{d-(j+1)}| & \lesssim \rho_{d}^{ -\frac{\eps}{d-1}}\Big(\frac{\rho_{d-(j+1)}}{\rho_{d-j}}\Big)^{d-(j+1)}\Big(\frac{r_{d-(j+1)}}{r_{d-j}}\Big)^{n-d+(j+1)}\\
&\qquad\qquad\cdot\Big(
C_j(n,E,K,\eps)\rho_d^{-\frac{j\eps}{d-1}}\rho_{d-j}^{d-j} r_{d-j}^{n-d+j}\prod_{i=0}^{j-1} (\rho_{d-i}r_{d-i}^{-1})\Big)\\
&= C_j(n,E,K,\eps) \rho_{d}^{-\frac{(j+1)\eps}{d-1}}\rho_{d-(j+1)}^{d-(j+1)} r_{d-(j+1)}^{n-d+(j+1)}\prod_{i=0}^{(j+1)-1} (\rho_{d-i}r_{d-i}^{-1}).
\end{split}
\end{equation}
The induction closes if we select $C_{j+1}(n,E,K,\eps)$ sufficiently large depending on $C_j(n,E,K,\eps),n,K,E$ and $\eps$.

To finish the proof, observe that
$$
|\tilde S_1|\leq C_{d-1}(n,E,K,\eps)\rho_d^{-\eps} r_{1}^{n}\prod_{i=0}^{d-1} (\rho_{d-i}r_{d-i}^{-1}).
$$
The result now follows from \eqref{boundOnTubesSzS1}.
\end{proof}

Next we will show how Lemma \ref{CountingTubesInsideSetsLem} can be used to prove Theorem \ref{polyWolffFlagsVarities}.
\begin{proof}[Proof of Theorem \ref{polyWolffFlagsVarities}]
Let $Z_1\supset \cdots\supset Z_d$, $1\geq r_1\geq\ldots\geq r_d>0$, $\delta>0,$ and $x\in\RR^n$ be as in the statement of Theorem \ref{polyWolffFlagsVarities}. Let $\tubes$ be a set of direction-separated $1\times \delta$ tubes, each of which satisfies
\begin{equation}\label{tubeIntersectsEachGrain}
|T\cap N_{2\delta}(Z_i) \cap B(x,r_i)| \geq r_i|T|,\ i=1,\ldots,d.
\end{equation}

Define $\rho_1=\ldots=\rho_d = \delta$. Since each variety $Z_i$ has codimension at least $i$, by Theorem \ref{WonkewThm}, each variety $Z_i$ satisfies the growth condition \eqref{SiGrowthCondition}. Let $\mathcal{L}$ be the set of lines coaxial with the tubes in $\tubes$. Observe that if $T\in\tubes$ is a tube with coaxial line $L$, then for each index $i=1,\ldots,d$ we have
$$
|T\cap N_{2\delta}(Z_i) \cap B(x,r_i)|/|T| \geq |L\cap N_{\delta}(Z_i) \cap B(x,r_i)|.
$$
In particular, if $T$ satisfies \eqref{tubeIntersectsEachGrain} then
\begin{equation}\label{lineIntersectsEachGrain}
|L\cap N_{\delta}(Z_i) \cap B(x,r_i)| \geq r_i,\ i=1,\ldots,d.
\end{equation}
Note that $L\cap N_{\delta}(Z_i) \cap B(x,r_i)$ is a one-dimensional semi-algebraic set of complexity $O(E)$, and thus if $K=K(n,E)$ is chosen sufficiently large, then for each index $i$, $L\cap N_{\delta}(Z_i) \cap B(x,r_i)$ contains a line segment of length at least $r_i/K$. 

To finish the proof, we apply Lemma \ref{CountingTubesInsideSetsLem} to $\lines$ with this choice of $K$.
\end{proof}

\begin{rem}
Lemma \ref{CountingTubesInsideSetsLem} was stated in greater generality than was needed to prove Theorem \ref{polyWolffFlagsVarities}. Specifically, in Theorem \ref{polyWolffFlagsVarities} we have $\rho_1=\ldots=\rho_d$. Another interesting situation occurs when $r_1>r_2>\ldots> r_d$ are much larger than 1, and $\rho_i$ is comparable to $r_i^{1/2}$; the motivation for this setup is as follows. In \cite{G2, G3}, Guth initiated the program of using polynomial partitioning to study the restriction/extension problem. When using polynomial partitioning techniques to analyze the behavior of the extension operator, a difficult sub-problem arises when many wave packets are concentrated near a low-degree variety (or more generally, wave packets are concentrated near ``grains,'' which are defined in Section \ref{multiLevelGrainsSec} below). In \cite{HR1} , Hickman and Rogers systematically studied the behavior of wave packets that concentrate near grains, and by using Theorem \ref{polyWolffAxiomsThm} they were able to obtain improved bounds for the restriction problem. Lemma \ref{CountingTubesInsideSetsLem} was used in \cite{HZ} to obtain further improvements. 
\end{rem}

\section{A multilevel grains decomposition}\label{multiLevelGrainsSec}
In this section, we will apply the polynomial partitioning theorem proved by Guth and Katz in \cite{GK} (and adapted to the present context by Guth in \cite{G3}) to collections of $1\times\delta$ tubes in $\RR^n$. We will establish a sort of dichotomy asserting that either (A): $\RR^n$ can be partitioned into disjoint pieces, and each tube is localized to a small number of these pieces, or (B): the tubes cluster into thin neighborhoods of low degree algebraic varieties. This statement will be made precise in Proposition \ref{multiLevelGrains} below. The results in this section do not make any assumptions about the directions of the tubes. In Section \ref{multilinearKakeyaSec}, we will apply Proposition \ref{multiLevelGrains} to sets of tubes pointing in different directions.
\begin{defn}
Let $\mathcal{G}$ be a set of tuples of the form $(B,P_1,\ldots,P_i)$, where $B\subset\RR^n$ is a ball, $i\geq 0,$ and each $P_j$ is a polynomial in $\RR^n$. $\mathcal{G}$ is called a \emph{tree of grains} if it satisfies the following properties.
\begin{itemize}
\item There is exactly one element $G_{\operatorname{root}}\in \mathcal{G}$ for which $i=0$.
\item If $(B,P_1,\ldots,P_i)$ and $(B^\prime,P_1,\ldots,P_i)$ are elements of $\mathcal{G}$, then $B\cap B^\prime=\emptyset$. 
\item For every $(B,P_1,\ldots,P_i)\in\mathcal{G},$ there is a ball $B^\prime$ containing $B$ so that $(B^\prime,P_1,\ldots,P_{i-1})\in\mathcal{G}$. 
\item For every $(B,P_1,\ldots,P_i)\in \mathcal{G},$ the variety $Z(P_1,\ldots,P_i)$ has codimension at least $i$.
\end{itemize} 
\end{defn}
If $G=(B,P_1,\ldots,P_i)\in \mathcal{G}$, we say that $G$ has level $i$, or $\operatorname{level}(G)=i$.  For each $0\leq i^\prime \leq i$, we define $G|_{i^\prime}$ to be the (unique) element of $\mathcal{G}$ of the form $(B^\prime,P_1,\ldots,P_{i^\prime})$ with $B\subset B^\prime$. If $G,G^\prime\in\mathcal{G}$ and $G$ has level $i$, we write $G\preceq G^\prime$ if $G^\prime = G|_{i^\prime}$ for some $0\leq i^\prime\leq i$. The relation $\preceq$ defines the natural partial order on $\mathcal{G}$ that arises from its tree structure. Note that $G \preceq G_{\operatorname{root}}$ for every $G\in\mathcal{G}$.

We define the depth of a tree $\mathcal{G}$ to be one less than the maximum length of a chain in $\mathcal{G}$. For example, if $\mathcal{G}=\{G_{\operatorname{root}}\}$, then $\mathcal{G}$ has depth 0. We define the complexity of $\mathcal{G}$ to be the maximum degree of any polynomial appearing in any tuple in $\mathcal{G}$. In practice, our trees will always have finite cardinality, so these quantities will always be finite.

\begin{defn}\label{compatibleWithTree}
Let $\mathcal{G}$ be a tree of grains. For each $G\in\mathcal{G}$, let $X^G\subset\RR^n$. We say that the set system $\{X^G\}_{G\in\mathcal{G}}$ is \emph{compatible} with $\mathcal{G}$ if the following holds.
\begin{itemize}
\item For each $G=(B,P_1,\ldots,P_i)\in\mathcal{G}$, $X^G\subset B\cap N_{C\delta}\big(Z(P_1,\ldots,P_i)\big)$, where $C>0$ is a constant (see Remark \ref{compatibleRemark} below).  

\item If $G,G^\prime\in\mathcal{G}$ and $G \preceq G^\prime$, then $X^G\subset X^{G^\prime}$.

\item If $G,G^\prime\in\mathcal{G}$ are not comparable, then then $X^G\cap X^{G^\prime}=\emptyset$.
\end{itemize} 
\end{defn}

\begin{rem}\label{compatibleRemark}
Definition \ref{compatibleWithTree} is (intentionally) slightly ambiguous, since we have not specified the constant $C$. In this paper we can take $C = 3\sqrt n$. 
\end{rem}

In this section we will be interested in the interactions between collections of tubes and certain discretized subsets of $\RR^n$ that we will call $\delta$-cubes.

\begin{defn}
We define a $\delta$-cube to be a set of the form $v + [0,\delta]^n$, where $v\in (\delta\ZZ)^n$. In particular, any two distinct $\delta$-cubes have disjoint interiors.
\end{defn}

The main result of this section is the following multilevel grains decomposition for families of $1\times\delta$ tubes in $\RR^n$. 

\begin{prop}\label{multiLevelGrains}
Let $2\leq k\leq n$, let $\tubes_1,\ldots,\tubes_k$ be sets of $1\times\delta$ tubes contained in $B(0,1)\subset \RR^n$, and let $Y$ be a set of tuples $(Q,T_1,\ldots,T_k)$, where $Q$ is a $\delta$-cube, and $T_j\in\tubes_j$ with $T_j\cap Q\neq\emptyset$ for each index $j$. Then for each $0\leq m\leq n$ and each $\eps>0$, there exists:
\begin{itemize}
\item A tree $\mathcal{G}$ of grains of depth $m$ and complexity $E(n,\eps)$.
\item For each index $j$ and each $T_j\in\tubes_j$, a set system $\{T_j^G\}_{G\in\mathcal{G}}$ that is compatible with $\mathcal{G}$.
\item For each index $0\leq i\leq m$ and each index $1\leq j\leq k$, a length $\delta\leq \ell_{i,j}\leq 1.$ 
\end{itemize}

These objects have the following properties. 
\begin{enumerate}
\myitem{M1}\label{tubesUniformLength} {\bf The tubes have uniform length.} For each $T_j\in\tubes_j$ and each $G\in\mathcal{G}$ with $\operatorname{level}(G)=i$, $T_j^G$ is a (possibly empty) disjoint union of $\ell_{i,j}\times\delta$ tubes.
\myitem{M2}\label{grainsLocalized}{\bf The grains are localized.} If $G=(B,P_1,\ldots,P_i)\in\mathcal{G}$, then $B$ has radius at most $4\max_{1\leq j\leq k}\ell_{i,j}+\delta$.
\myitem{M3}\label{tubesTouchFewGrains}{\bf The tubes touch few grains.} For each index $j$ and each $G\in\mathcal{G}$ with $\operatorname{level}(G)=i$, we have
\begin{equation}\label{tubesTouchFewGrainsEquation}
M^{-\eps} (D_1^{1-n}\cdots D_i^{i-n})(\#\tubes_j)\lesssim_{\eps} \sum_{T_j\in\tubes_j} \#\operatorname{CC}(T_j^G)\lesssim_{\eps} M^{\eps}(D_1^{1-n}\cdots D_i^{i-n})(\#\tubes_j),
\end{equation}
where
\begin{equation}\label{defnM}
M = \delta^{-1}\prod_j(\#\tubes_j),
\end{equation}
and the real numbers $D_1,\ldots,D_m$ are defined so that $\mathcal{G}$ has $D_i^{n-i+1}$ grains that have level $i$.
\myitem{M4}\label{mostOfYCaptured} {\bf $Y$ is evenly distributed over the tree.} For each $G\in\mathcal{G}$ with $\operatorname{level}(G)=i$, we have
\begin{equation}\label{YGBddAbove}
\begin{split}
&M^{-\eps} (D_1^{-n}\cdots D_i^{i-n-1})
\sum_{Q} \big(\#\{
(Q, T_1,\ldots,T_k)\in Y\}\big)^{{\frac{1}{k-1}}}\\
&\qquad \sum_{Q} \big(  \#\{
(Q, T_1,\ldots,T_k)\in  Y \colon 
Q\cap T_j^G \neq\emptyset\ \textrm{for each index}\ j\}\big)^{\frac{1}{k-1}}
\\
&\qquad\qquad\lesssim_{\eps} \ M^{\eps} (D_1^{-n}\cdots D_i^{i-n-1})
\sum_{Q} \big(\#\{
(Q, T_1,\ldots,T_k)\in Y\}\big)^{{\frac{1}{k-1}}}.
\end{split}
\end{equation}
\end{enumerate}
\end{prop}

We will prove Proposition \ref{multiLevelGrains} by repeatedly applying a ``grains decomposition'' type result. This result has two main steps, which are described in Lemmas \ref{singleGrainDecompFirstStep} and \ref{singleGrainDecompSecondStep} below. In Lemma \ref{singleGrainDecompFirstStep}, we will use polynomial partitioning to divide the set of cubes supporting $Y$ into disjoint regions, so that every tube interacts with only a small number of these regions. The ideas used in this step are not new; they first appeared in \cite{G3} in the context of the restriction problem, and in \cite{GZ} in the context of Kakeya. 

In \cite{G3}, Guth proved a $k$-broad estimate for the restriction / extension operator. This is a weaker variant of a $k$-linear restriction estimate. In Section 11.1 of \cite{G3}, Guth discusses a shortcoming of his methods that prevents him from proving a $k$-linear restriction estimate. We encountered a similar problem when attempting to prove $k$-linear Kakeya estimates, and Lemma \ref{singleGrainDecompSecondStep} is designed to overcome this problem.

Before we prove Lemmas \ref{singleGrainDecompFirstStep} and \ref{singleGrainDecompSecondStep}, we will recall some terminology and results from \cite{G3}. If $P_1,\ldots,P_m$ are polynomials in $\RR^n$, we say that $Z(P_1,\ldots,P_m)$ is a transverse complete intersection if the vectors $\nabla P_1(x),\ldots,\nabla P_m(x)$ are linearly independent for every $x\in Z(P_1,\ldots,P_m)$. In particular, if $Z(P_1,\ldots,P_m)$ is a transverse complete intersection, then $Z(P_1,\ldots,P_m)$ is a smooth submanifold of $\RR^n$ of codimension $m$. 

In \cite{GZ}, Guth and Katz used the polynomial ham sandwich theorem to construct polynomials that efficiently partition sets of points in $\RR^n$. In \cite{G3}, Guth adapted these methods to prove the following.
\begin{prop}\label{algCellularDichotomyProp}
Let $P_1,\ldots,P_m$ be polynomials in $\RR^n$, and suppose that $Z(P_1,\ldots,P_m)$ is a transverse complete intersection. Let $f\in L^1(\RR^n)$ be non-negative, and suppose that $\operatorname{supp}(f) \subset B(0,1)\cap N_{2\delta}\big(Z(P_1,\ldots,P_m)\big)$. Then for each $D\geq 1$, at least one of the following two things must happen.

\medskip

\noindent{\bf Cellular case.}
There exists a polynomial $P$ of degree $\leq D$, so that $\RR^n\backslash Z(P)$ is a union of $\lesssim D^{n-m}$ cells $O_1,\ldots,O_t$, so that if  we define $O_i^\prime = O_i\backslash N_{\delta}(Z(P))$, then 
$$
\int_{O_i^\prime}f \lesssim D^{m-n}\int f\quad\textrm{for each index}\ i,
$$
and
$$
\sum_{i=1}^t \int_{O_i^\prime}f  \geq\frac{1}{2} \int f.
$$
\medskip

\noindent{\bf Algebraic case.} There is a polynomial $P$ of degree at most $D$ so that $(P_1,\ldots,P_m,P)$ is a transverse complete intersection, and
$$
\int_{N_{\delta}(Z(P_1,\ldots,P_m,P))}f \gtrsim (\log D)^{-1}\int f.
$$
\end{prop}
\begin{cor}\label{decompositionCubes}
Let $P_1,\ldots,P_m$ be polynomials in $\RR^n$, and suppose that $Z(P_1,\ldots,P_m)$ is a transverse complete intersection. Let $\mathcal{Q}$ be a finite set of $\delta$-cubes. Suppose that $Q \subset B(0,1)\cap N_{2\sqrt n \delta}\big(Z(P_1,\ldots,P_m)\big)$ for each $Q\in\mathcal{Q}$. Then for each $D\geq 1$, at least one of the following two things must happen.

\medskip

\noindent{\bf Cellular case.}
There exists a polynomial $P$ of degree $\leq D$, so that $\RR^n\backslash Z(P)$ is a union of $\lesssim D^{n-m}$ cells $O_1,\ldots,O_t$, so that if  we define 
\begin{equation}\label{defnQi}
\mathcal{Q}_i = \{Q\in\mathcal{Q}\colon Q\subset O_i \backslash N_{\delta}(Z(P))\},
\end{equation}
then 
\begin{equation}\label{cellularCaseEachCellCor}
\#\mathcal{Q}_i \lesssim D^{m-n} (\#\mathcal{Q}) \quad\textrm{for each index}\ i,
\end{equation}
and
\begin{equation}\label{cullularCaseCorWeight}
\sum_{i=1}^t \#\mathcal{Q}_i  \geq\frac{1}{2} (\#\mathcal{Q}).
\end{equation}
\medskip

\noindent{\bf Algebraic case.} There is a polynomial $P$ of degree $\leq D$ so that $(P_1,\ldots,P_m,P)$ is a transverse complete intersection, and if we define
\begin{equation}
\mathcal{Q}^*=\{Q\in\mathcal{Q}\colon Q\subset N_{2\sqrt n\delta}\big(Z(P_1,\ldots,P_m,P)\big)\},
\end{equation}
then
\begin{equation}\label{algCaseCorollary}
\#\mathcal{Q}^*\gtrsim (\log D)^{-1}(\#\mathcal{Q}).
\end{equation}
\end{cor}
\begin{rem}
Note that since each $\delta$-cube in $\RR^n$ has diameter $\sqrt n\delta$, if a cube $Q$ is not contained in any cell $O_i^\prime$ then it must intersect the $\delta$-neighborhood of $\RR^n\backslash\bigcup O_i$, and thus it must be contained in the $(\sqrt{n}+1)\delta$-neighborhood of $\RR^n\backslash\bigcup O_i$. If this happens for at least half the cubes in $\mathcal{Q}$, then it is possible to find a polynomial $P$ that satisfies \eqref{algCaseCorollary}.
\end{rem}

\begin{lem}\label{singleGrainDecompFirstStep}
Let $P_1,\ldots,P_m$ be polynomials in $\RR^n$, and suppose that $Z=Z(P_1,\ldots,P_m)$ is a transverse complete intersection. Let $\mathcal{Q}$ be a finite set of $\delta$-cubes in $\RR^n$ that are contained in $B(0,1)\cap N_{2\sqrt n\delta}(Z)$. Let $\eps>0$.

Then there is a set $\mathcal{A}=\{(P_A,\mathcal{Q}_A)\}$, so that that sets $\{\mathcal{Q}_A\}_{A\in\mathcal{A}}$ are disjoint subsets of $\mathcal{Q}$, and
for each $(P_A,\mathcal{Q}_A)\in \mathcal{A}$ we have, 
\begin{enumerate}
\myitem{C1}\label{C1Item} $P_A$ is a polynomial of degree at most $E(n,\eps)$.
\myitem{C2}\label{C2Item} $Z(P_1,\ldots,P_m,P_A)$ is a transverse complete intersection.
\myitem{C3}\label{C3Item} Each $\delta$-cube in $\mathcal{Q}_A$ is contained in $N_{2\sqrt n\delta}(Z(P_1,\ldots,P_m,P_A))$.
\myitem{C4}\label{C4Item}\itemizeEqnVSpacing
\begin{equation}\label{boundOnEachGrainSize}
(\#\mathcal{A})^{-1-\eps}(\#\mathcal{Q}) \lesssim_{\eps}\#\mathcal{Q}_A \lesssim_{\eps}  (\#\mathcal{A})^{-1+\eps}(\#\mathcal{Q}).
\end{equation}
\end{enumerate}

Furthermore, if $T$ is a tube of thickness $\delta$ (and any length), then there are disjoint subsets $\{T^{A}\}_{A\in\mathcal{A}}$ of $T$ with the following properties
\begin{enumerate} 
\myitem{T1}\label{t1Item} $T^A\subset N_{3\sqrt n \delta}\big(Z(P_1,\ldots,P_m,P_A)\big)$ for each $A\in\mathcal{A}$.
\myitem{T2}\label{T2Item}  If $(P_A,\mathcal{Q}_A)\in\mathcal{A},$ $Q\in\mathcal{Q}_A$, and $Q\cap T\neq\emptyset$, then $Q\cap T^A\neq\emptyset$. 
\myitem{T3}\label{T3Item} Each set $T^A$ is either empty or a union of disjoint sub-tubes of $T$. 
\myitem{T4}\label{T4Item} The number of sub-tubes in $T^A$, summed across all $A\in\mathcal{A}$, is not too big. Specifically, we have

\begin{equation}\label{sumOfSubTubes}
\sum_{A\in\mathcal{A}}\#\operatorname{CC}(T^A)\lesssim_{\eps}(\#\mathcal{A})^{1/(n-m)+\eps}.
\end{equation}
\end{enumerate}
\end{lem}
\begin{proof}
This lemma is not new; a variant of this lemma first appeared in \cite{G3} in the context of the restriction problem, and in \cite{GZ} in the context of Kakeya. 

Let $E=E(\eps,n)$ be a large number to be chosen later. We will repeatedly apply Corollary \ref{decompositionCubes} to construct a tree $\mathcal{T}$ as follows. The root of $\mathcal{T}$ will be the pair $(\RR^n,\mathcal{Q})$. The non-leaf vertices of $\mathcal{T}$ will be pairs $(O,\mathcal{Q}_O)$, where $O$ is an open subset of $\RR^n$ and $\mathcal{Q}_O\subset\mathcal{Q}$. The leaf vertices of $\mathcal{T}$ will be pairs $(P,\mathcal{Q}_P)$, where $P$ is a polynomial in $\RR^n$ of degree $\leq E$, and $\mathcal{Q}_P\subset\mathcal{Q}$. 

If $(O,\mathcal{Q}_O)$ is a non-leaf vertex, then $Q\subset O \backslash N_{\delta}(\operatorname{bdry}(O))$ for each $Q\in\mathcal{Q}_O$. If $(O^\prime,\mathcal{Q}_{O^\prime})$ is the parent of $(O,\mathcal{Q}_O)$, then $O\subset O^\prime$ and $\mathcal{Q}_O\subset\mathcal{Q}_{O^\prime}$. If $(P,\mathcal{Q}_P)$ is a leaf vertex, then $Q\subset N_{2\sqrt n\delta}(Z(P_1,\ldots,P_m,P))$ for each $Q\in\mathcal{Q}_P$. If $(O,\mathcal{Q}_{O})$ is the parent of $(P,\mathcal{Q}_P)$, then $\mathcal{Q}_P\subset\mathcal{Q}_{O}$.

For each non-leaf vertex $(O,\mathcal{Q}_O)$ of $\mathcal{T}$, exactly one of the following two things must be true.

\medskip

\noindent{\bf Case 1: Algebraic leaf}. $(O,\mathcal{Q}_O)$ has one child, which is a leaf of the form $(P,\mathcal{Q}_P$). Furthermore, $Z(P_1,\ldots,P_m,P)$ is a transverse complete intersection,
\begin{equation}\label{mostMassInLeaf}
\#\mathcal{Q}_P \gtrsim (\log E)^{-1} (\#\mathcal{Q}_O),
\end{equation}
and
\begin{equation}\label{cubesCloseToLeaf}
Q\subset N_{2\sqrt n\delta}Z(P_1,\ldots,P_m,P)\quad\textrm{for each}\ Q\in\mathcal{Q}_P.
\end{equation}

\medskip

\noindent{\bf Case 2: Cellular vertex}. $(O,\mathcal{Q}_O)$ has $\sim E^{n-m}$ children $(O_1,\mathcal{Q}_{O_1}),\ldots, (O_t,\mathcal{Q}_{O_t})$. Furthermore,
\begin{equation}\label{massSpreadOverCells}
\#\mathcal{Q}_{O_i} \lesssim  E^{m-n}(\#\mathcal{Q}_O)\quad\textrm{for each index}\ i,
\end{equation}
and
\begin{equation}\label{mostMassCapturedByCells}
\sum_{i=1}^t(\#\mathcal{Q}_{O_i})\geq \frac{1}{2}(\#\mathcal{Q}_O).
\end{equation}
Finally, the sets $O_1,\ldots,O_t$ are disjoint.

Observe that the maximum depth of $\mathcal{T}$ is $\lesssim \log_E(\#\mathcal{Q})$, since each set $\mathcal{Q}_O$ at level $s$ contains at least one, but at most $\big(C(n) E^{m-n}\big)^s(\#\mathcal{Q})$ cubes. This means that 
\begin{equation}
\sum_{(P,\mathcal{Q}_P)\in\mathcal{T}}(\#\mathcal{Q}_P) \gtrsim 2^{-\log_E(\#\mathcal{Q})}(\#\mathcal{Q}),
\end{equation}
where the sum is taken over all leafs of the tree. In particular, there is a level $s_0$ so that 
\begin{equation}\label{mostMassCapturedHere}
\sum_{\substack{(P,\mathcal{Q}_P)\\ \operatorname{level}(P,\mathcal{Q}_P) = s_0}}(\#\mathcal{Q}_P) \gtrsim 2^{-\log_E(\#\mathcal{Q})} (\log E)^{-1}(\#\mathcal{Q}).
\end{equation}
If $E$ is chosen sufficiently large depending on $\eps$, then 
$$
2^{-\log_E(\#\mathcal{Q})} (\log E)^{-1}\gtrsim_{\eps}(\#\mathcal{Q})^{-\eps/2}.
$$
Define $\mathcal{A}^\prime$ to be the set of pairs $(P,\mathcal{Q}_P)$ that have level $s_0$. Then \eqref{mostMassCapturedHere} implies that
\begin{equation}\label{mostMassCaptured}
\sum_{A\in\mathcal{A}^\prime}(\#\mathcal{Q}_A)\gtrsim_{\eps}(\#\mathcal{Q})^{1-\eps/2}.
\end{equation}

By \eqref{massSpreadOverCells}, we have
\begin{equation}\label{notTooMuchMassEachCell}
(\#\mathcal{Q}_A)\leq \Big(C(n) E^{m-n}\Big)^{s_0}(\#\mathcal{Q})\quad\textrm{for each}\ A\in\mathcal{A}^\prime.
\end{equation}
Since
\begin{equation}\label{boundOnCardA}
\#\mathcal{A}\geq \Big(C(n)^{-1}E^{n-m}\Big)^{s_0},
\end{equation}
we have that if $E$ is chosen sufficiently large depending on $\eps$, then \eqref{notTooMuchMassEachCell} and \eqref{boundOnCardA} imply
\begin{equation}\label{eachGrainSmall}
\#\mathcal{Q}_A \lesssim_{\eps}(\#\mathcal{A}^\prime)^{-1+\eps/2}(\#\mathcal{Q})\quad\textrm{for each}\ A\in\mathcal{A}^\prime.
\end{equation}
By pigeonholing, \eqref{mostMassCaptured} and \eqref{eachGrainSmall} imply that there exists a set $\mathcal{A}\subset\mathcal{A}^\prime$ with $\#\mathcal{A}\gtrsim_{\eps} (\#\mathcal{A}^\prime)^{1-\eps}$ that satisfies \eqref{boundOnEachGrainSize}.


It remains to analyze how a tube interacts with the elements of $\mathcal{A}$. Observe that if $(O,\mathcal{Q}_O)$ is a vertex of the tree with children $(O_1,\mathcal{Q}_{O_1}),\ldots, (O_t,\mathcal{Q}_{O_t})$, and if $T$ is a tube of thickness $\delta$ (and any length) contained in $O$, then the set 
$
\bigcup_{i=1}^t \big(T\cap (O_i\backslash N_{\delta}(\operatorname{bdry}(O_i))\big)
$
is contained in a union of $\leq E+1$ sub-tubes of $T$, each of which is contained in a cell $O_i$. If $T\cap Q\neq\emptyset$ for some $Q\in\mathcal{Q}_{O_i}$, then at least one of these sub-tubes intersects $Q$. Thus if $T$ is a tube of thickness $\delta$, then for each non-leaf vertex $(O,\mathcal{Q}_O)$ of $\mathcal{T}$ we can associate a set $T^O\subset T$, which is a disjoint union of sub-tubes of $T$ contained in $O$. 

If $(P,\mathcal{Q}_P)$ is a leaf of $\mathcal{T}$ with parent $(O,\mathcal{Q}_O)$, and if $T$ is a tube of thickness $\delta$ (and any length) contained in $O$, then there is a set $T^P\subset T$ that is a disjoint union of $O_E(1)$ sub-tubes of $T$, each contained in $N_{3\sqrt n\delta}\big(Z(P_1,\ldots,P_m,P)\big)$. The set $T^P$ can be constructed by taking the smallest union of sub-tubes of $T$ that contains $T\cap N_{2\sqrt n\delta}(Z(P_1,\ldots,P_m,P))$. By Theorem \ref{BPRTheorem}, $T\cap N_{2\sqrt n\delta}\big(Z(P_1,\ldots,P_m,P)\big)$ has $O_E(1)$ connected components, and thus $T^P$ will be a union of $O_E(1)$ sub-tubes of $T$. The set $T^P$ has the property that if $Q\in\mathcal{Q}_P$ and $T\cap Q\neq\emptyset$, then $T^P\cap Q\neq\emptyset$.

For each level $s\geq 0$ of the tree, we have that $\bigcup_{P}T^P$ is a disjoint union of disjoint sub-tubes of $T$, and this union contains at most $O_E(1)(E+1)^s$ sub-tubes (the union is taken over all leafs of the tree that have level $s$). In particular, setting $s=s_0$ we see that $\bigcup_{A\in\mathcal{A}}T^A$ is a disjoint union of disjoint sub-tubes of $T$, and this union contains at most $O_E(1) (E+1)^{s_0}$ sub-tubes. If $E$ is chosen sufficiently large (depending on $\eps$), then this quantity is $\lesssim_{\eps}(\#\mathcal{A})^{\frac{1}{n-m}+\eps}$, which establishes Property \ref{T4Item}. Finally, if $A\in\mathcal{A}$, $Q\in\mathcal{Q}_A$, and $T\cap Q\neq\emptyset$, then $T^A\cap A\neq\emptyset$. This establishes Property \ref{T2Item}.
\end{proof}

\begin{lem}\label{singleGrainDecompSecondStep}
Let $2\leq k\leq n$ and let $\lambda_1,\ldots,\lambda_k$ be real numbers between $\delta$ and $1$. For each $j=1,\ldots,k$, let $\tubes_j$ be a set of $\lambda_j\times\delta$ tubes in $\RR^n$. Let $P_1,\ldots,P_m$ be polynomials and let $B_0\subset\RR^n$ be a ball of radius at most $\max_{1\leq j\leq k}\lambda_j+\delta$. Suppose that $Z(P_1,\ldots,P_m)$ is a transverse complete intersection, and that for each index $j$, each tube $T_j\in\tubes_j$ is contained in $B_0\cap N_{3\sqrt n\delta}\big(Z(P_1,\ldots,P_m)\big)$. Let $Y$ be a set of tuples $(Q,T_1,\ldots,T_k)$, where $Q$ is a $\delta$-cube contained in $B_0\cap N_{2\sqrt n\delta}(Z(P_1,\ldots,P_n))$, and $T_j\in\tubes_j$ is a tube intersecting $Q$.

Then for each $\eps>0$, there exists
\begin{itemize}
\item A set $\mathcal{H}$ of triples $H=(B_H,P_H,\mathcal{Q}_H)$, where $B_H\subset B_0$ is a ball, $P_H$ is a polynomial of degree at most $E(n,\eps)$, and $\mathcal{Q}_H$ is a set of $\delta$-cubes.
\item For each index $j$ and each $T_j\in\tubes_j$, a collection $\{T_j^H\}_{H\in\mathcal{H}}$ of subsets of $T_j$.
\item Lengths $\ell_1,\ldots,\ell_k$, each of the form $\delta 2^N$ for some integer $N\geq 0$.
\end{itemize}

These objects have the following properties. 
\begin{enumerate}
\myitem{G1}\label{P1PmCompleteIntersection}  For each $H=(B_H,P_H,\mathcal{Q}_H)\in\mathcal{H}$, we have that $Z(P_1,\ldots,P_m,P_H)$ is a transverse complete intersection and $Q\subset N_{2\sqrt n\delta}(Z(P_1,\ldots,P_m,P_H))$ for each $Q\in\mathcal{Q}_H$.

\myitem{G2}\label{areDisjoint} The sets $\{\mathcal{Q}_H\}_{H\in\mathcal{H}}$ are disjoint.

\myitem{G3}\label{tubesContainedNeighborhoodGrains} For each index $j$, each $T_j\in\tubes_j$, and each $H=(B_H,P_H,\mathcal{Q}_H)\in\mathcal{H}$, we have the containment $T_j^H\subset B_H\cap N_{3\sqrt n\delta}\big(Z(P_1,\ldots,P_m,P_H)\big)$.

\myitem{G4}\label{tubeSegmentLength} For each index $j$, each $T_j\in\tubes_j$, and each $H\in\mathcal{H}$, we have that $T^H$ is a (possibly empty) disjoint union of sub-tubes of $T$, each of which has length $\ell_j$. The sets $\{T^H\}_{H\in\mathcal{H}}$ are disjoint.

\myitem{G5}\label{diamOfBall} For each $(B_H,P_H,\mathcal{Q}_H)\in\mathcal{H}$, $B_H$ has radius at most $4\max_{1\leq j\leq k}\ell_{j}+\delta$.

\myitem{G6}\label{numberOfTubeSegments} For each index $j$ and each $H\in\mathcal{H}$,
\begin{equation}\label{numberTubeSegmentsTj}
(\#\tubes_j) (\#\mathcal{H})^{\frac{1}{n-m}-1} \leq \sum_{T_j\in\tubes_j}\#\operatorname{CC}(T_j^H) \lesssim_{\eps} M^{\eps} (\#\tubes_j) (\#\mathcal{H})^{\frac{1}{n-m}-1},
\end{equation}
where
\begin{equation}\label{defnMRd2}
M = \delta^{-1}\prod_j(\#\tubes_j).
\end{equation}

\myitem{G7}\label{smallYInEachGrain} For each $H\in\mathcal{H}$ we have
\begin{equation}\label{YGBddAboveBelow}
\begin{split}
 &M^{-\eps}(\#\mathcal{H})^{-1}
\sum_{Q} \big(\#\{
(Q, T_1,\ldots,T_k)\in Y\}\big)^{{\frac{1}{k-1}}}\\
&\qquad\lesssim \sum_{Q} \big(  \#\{
(Q, T_1,\ldots,T_k)\in  Y \colon 
Q\in \mathcal{Q}_H,\ Q\cap T_j^H \neq\emptyset\ \textrm{for each index}\ j\}\big)^{\frac{1}{k-1}}\\
&\qquad\qquad\lesssim_{\eps}
M^{\eps} (\#\mathcal{H})^{-1}
\sum_{Q} \big(\#\{
(Q, T_1,\ldots,T_k)\in Y\}\big)^{{\frac{1}{k-1}}}.
\end{split}
\end{equation}

\end{enumerate}
\end{lem}
\begin{proof}
After dyadic pigeonholing, we can find a set $Y^\prime\subset Y$ and a set $\mathcal{Q}$ of $\delta$-cubes so that $\# Y^\prime\gtrsim (\# Y)/\log(\# Y)$; for each $(Q,T_1,\ldots,T_k)\in  Y^\prime$ we have $Q\in\mathcal{Q}$; and 
\begin{equation}\label{uniformityOfYPrime}
\#\{ (T_1,\ldots,T_k)\colon (Q,T_1,\ldots,T_k)\in Y^\prime\} \sim (\# Y^\prime)/(\#\mathcal {Q}) \quad\textrm{for each}\ Q\in\mathcal{Q}.
\end{equation}
In particular, we have
$$
\sum_{Q\in\mathcal{Q}^\prime}\big(\#\{ (T_1,\ldots,T_k)\colon (Q,T_1,\ldots,T_k)\in Y^\prime\} \big)^{\frac{1}{k-1}}\sim (\#Y^\prime)^{\frac{1}{k-1}} (\#\mathcal{Q})^{\frac{k-2}{k-1}}.
$$

Apply Lemma \ref{singleGrainDecompFirstStep} to $\mathcal{Q}$ with allowable error $\eps/C,$ where $C$ is a large constant to be chosen later. Let $\mathcal{A}=\{(P_A,\mathcal{Q}_A)\}$ be the output from the lemma. By \eqref{uniformityOfYPrime} and the lower bound from \eqref{boundOnEachGrainSize}, we have  
\begin{equation}
\begin{split}
\sum_{Q\in\mathcal{Q}}\big( \#\{(Q,T_1,\ldots,T_k)\in Y^\prime \colon Q\cap\bigcup_{A\in\mathcal{A}}T_j^A\neq\emptyset\ & \textrm{for each index j}\}\big)^{\frac{1}{k-1}}\\
&\gtrsim_{\eps} (\#Y^\prime)^{\frac{1}{k-1}-O(\eps/C)} (\#\mathcal{Q})^{\frac{k-2}{k-1}}\\
&\gtrsim_{\eps} M^{-O(\eps/C)}(\#Y^\prime)^{\frac{1}{k-1}} (\#\mathcal{Q})^{\frac{k-2}{k-1}}.
\end{split}
\end{equation}
On the last line we used the fact that $\# Y^\prime\leq \#Y\leq M$. Indeed, this inequality motivates the definition of $M$.

For each index $j$ and each $T_j\in\tubes_j,$ we have that $\bigcup_{A\in\mathcal{A}}T_j^A$ is a disjoint union of $\lesssim_{\eps}(\#\mathcal{A})^{\frac{1}{n-m}+\frac{\eps}{C}}$ sub-tubes of $T_j$, each of which has length between $\delta$ and $\lambda_j$. Thus by dyadic pigeonholing, there are numbers $\ell_{j}^{(0)},\ j=1,\ldots,k,$ each of the form $\delta 2^{N}$ for some non-negative integer $N$,  so that
\begin{equation}\label{mostMassCapturedRound3}
\begin{split}
\sum_{Q\in\mathcal{Q}}  \Big(  \#\Big\{ &
(Q, T_1,\ldots,T_k)\in Y^\prime \colon 
Q\cap T_j\ \textrm{is contained in}\\ 
&\textrm{a sub-tube of}\ \bigcup_{A\in\mathcal{A}}T_j^A\ \textrm{of length between }\ell_j^{(0)}\ \textrm{and}\ 2\ell_j^{(0)}\Big\}\Big)^{\frac{1}{k-1}}\\
&\qquad\qquad\qquad\gtrsim_{\eps}(\log(1/\delta))^{-k} (\#Y^\prime)^{\frac{1}{k-1}-O(\eps/C)} (\#\mathcal{Q})^{\frac{k-2}{k-1}}.
\end{split}
\end{equation}

Recall that the sets $\{T_j^A\}_{A\in\mathcal{A}}$ are disjoint, and each of these sets is a disjoint union of sub-tubes of $T_j$. For each index $j$, each $T_j\in\tubes_j$, and each $A\in\mathcal{A}$, define $\tilde T_j^A$ to be the union of those sub-tubes in $T_j^A$ that have length between $\ell_j^{(0)}$ and $2\ell_j^{(0)}$. Thus \eqref{mostMassCapturedRound3} becomes

\begin{equation*}
\begin{split}
\sum_{Q\in\mathcal{Q}} \Big(  \#\big\{
(Q, T_1,\ldots,T_k)&\in  Y^\prime \colon 
Q\cap \bigcup_{A\in\mathcal{A}}\tilde T_j^A \neq\emptyset \big\}\Big)^{\frac{1}{k-1}}\\
&\gtrsim_{\eps} (\log(1/\delta))^{-k} (\#Y^\prime)^{\frac{1}{k-1}-O(\eps/C)} (\#\mathcal{Q})^{\frac{k-2}{k-1}}\\
&\gtrsim_{\eps} M^{-O(\eps/C)}(\#Y^\prime)^{\frac{1}{k-1}} (\#\mathcal{Q})^{\frac{k-2}{k-1}}.
\end{split}
\end{equation*}

Recall that for each of the tuples $(Q,T_1,\ldots,T_k)$ in the above sum, there is an element $A\in\mathcal{A}$ so that $Q\in\mathcal{Q}_A$ and $Q\cap \tilde T_j^A\neq\emptyset$ for each index $j$. Thus

\begin{equation*}
\begin{split}
\sum_{A\in\mathcal{A}}\sum_{Q\in\mathcal{Q}_A} \Big(  \#\big\{
(Q, T_1,\ldots,T_k)&\in  Y^\prime \colon 
Q\cap \tilde T_j^A \neq\emptyset \big\}\Big)^{\frac{1}{k-1}}\\
&\gtrsim_{\eps} M^{-O(\eps/C)} (\#Y^\prime)^{\frac{1}{k-1}} (\#\mathcal{Q})^{\frac{k-2}{k-1}}.
\end{split}
\end{equation*}
Define $r^{(0)}=2\max_j\ell_j^{(0)}$. Let $\mathcal{B}$ be a set of balls of radius $2r^{(0)}$ in $\RR^n$ so that the balls are $O(1)$ overlapping, and every $r^{(0)}\times\delta$ tube is contained in at least one of the balls. For each such ball $B$, define $\tilde T_j^{A,B}$ to be the union of the sub-tubes in $\tilde T_j^A$ that are contained in $B$. Then there is a subset $\mathcal{B}^\prime\subset\mathcal{B}$ consisting of disjoint balls so that
\begin{equation*}
\begin{split}
\sum_{A\in\mathcal{A}}\sum_{B\in\mathcal{B}^\prime}\sum_{Q\in\mathcal{Q}_A} \Big(  \#\big\{
(Q, T_1,\ldots,T_k)& \in  Y^\prime \colon 
Q\cap \tilde T_j^{A,B} \neq\emptyset\big\}\Big)^{\frac{1}{k-1}}\\
&\gtrsim_{\eps} M^{-O(\eps/C)}(\#Y^\prime)^{\frac{1}{k-1}} (\#\mathcal{Q})^{\frac{k-2}{k-1}}.
\end{split}
\end{equation*}
For each $A\in\mathcal{A}$ and each $B\in\mathcal{B}$, define
$$
\mathcal{Q}_{A,B}=\{Q\in\mathcal{Q}_A\colon Q\subset B\}.
$$
Define 
$$
\mathcal{H}^\prime=\bigcup_{A\in\mathcal{A}} \{ (B, P_A, \mathcal{Q}_{A,B})\colon B\in\mathcal{B}^\prime\}.
$$

If $H = (B_H,P_H,\mathcal{Q}_H)\in\mathcal{H}^\prime$ and $P_H=P_A$ for some $A\in\mathcal{A}$, then for each index $j$ and each $T_j\in\tubes_j$, define $T_j^H=\tilde T_j^{A,B}$; thus $T_j^H$ is a disjoint union of sub-tubes of $T_j$, each of which is contained in $B_H\cap N_{3\sqrt n\delta}(Z(P_1,\ldots,P_m,P))$ and has length between $\ell_{j}^{(0)}$ and $2\ell_{j}^{(0)}$.

We have
\begin{equation}\label{mostIncidencesStillCaptured}
\begin{split}
\sum_{H\in\mathcal{H}^\prime } \sum_{Q\in\mathcal{Q}_H} \Big(  \#\big\{
(Q, T_1,\ldots,T_k)&\in  Y^\prime \colon 
Q\cap T_j^H \neq\emptyset\big\}\Big)^{\frac{1}{k-1}}\\
&\gtrsim_{\eps}M^{-O(\eps/C)}(\#Y^\prime)^{\frac{1}{k-1}} (\#\mathcal{Q})^{\frac{k-2}{k-1}}.
\end{split}
\end{equation}
Abusing notation slightly, we will re-define the sets $T_j^H$ so that each such set is a disjoint union of sub-tubes of $T_j$ of length precisely $\ell_j^{(0)}$. If we choose these sub-tubes appropriately, then \eqref{mostIncidencesStillCaptured} remains true (though the quasi-inequality is weakened by a constant factor). This step is not crucial for our proof; we do it only for notational convenience later on.

For each $H\in\mathcal{H}^\prime$, we will be interested in the quantity
$$
\sum_{Q\in\mathcal{Q}_H} \Big(  \#\big\{
(Q, T_1,\ldots,T_k)\in  Y^\prime \colon  Q\cap T_j^H \neq\emptyset\big\}\Big)^{\frac{1}{k-1}}.
$$
Note that this number is of the form $N^{\frac{1}{k-1}}$, where $N$ is an integer between 0 and $\#Y\leq M$. We will also be interested in the numbers
$$
\sum_{T_j\in\tubes_j}\#\operatorname{CC}(T_j^{H}),
$$
which are non-negative integers bounded by $M$ (this follows from the fact that each connected component of $T_j^H$ has length $\ell_j^{(0)}\geq\delta$). After dyadic pigeonholing, we can find a subset $\mathcal{H}^{(0)}\subset\mathcal{H}^\prime$ so that the above quantities are roughly the same for each $H\in\mathcal{H}^{(0)}$. Specifically, the following items hold
\begin{itemize}
\item \itemizeEqnVSpacing\itemizeEqnVSpacing
\begin{equation}\label{mostMassPreserved}
\begin{split}
\sum_{H\in\mathcal{H}^{(0)}} & \sum_{Q\in\mathcal{Q}_H} \Big(  \#\big\{
(Q, T_1,\ldots,T_k)\in  Y^\prime \colon 
Q\cap T_j^H \neq\emptyset\big\}\Big)^{\frac{1}{k-1}} \\
&\gtrsim (\log M)^{-1} \sum_{H\in\mathcal{H}^\prime } \sum_{Q\in\mathcal{Q}_H} \Big(  \#\big\{
(Q, T_1,\ldots,T_k)\in  Y^\prime \colon 
Q\cap T_j^H \neq\emptyset\big\}\Big)^{\frac{1}{k-1}}.
\end{split}
\end{equation}

\item For each $H_0\in\mathcal{H}^{(0)},$
\begin{equation}\label{evenlySplitBetweenH0}
\begin{split}
\sum_{Q\in\mathcal{Q}_{H_0}} &\Big(  \#\big\{
(Q, T_1,\ldots,T_k)\in  Y^\prime \colon 
Q\cap T_j^{H_0} \neq\emptyset\big\}\Big)^{\frac{1}{k-1}}\\
& \sim (\#\mathcal{H}^{(0)})^{-1}\sum_{H\in\mathcal{H}^{(0)}}  \sum_{Q\in\mathcal{Q}_H} \Big(  \#\big\{
(Q, T_1,\ldots,T_k)\in  Y^\prime \colon 
Q\cap T_j^H \neq\emptyset\big\}\Big)^{\frac{1}{k-1}}.
\end{split}
\end{equation}

\item For each $H_0\in\mathcal{H}^{(0)}$ and each index $j$,
\begin{equation}\label{equalNumbersTubesEachCell}
\sum_{T_j\in\tubes_j}\#\operatorname{CC}(T_j^{H_0}) \sim (\#\mathcal{H}^{(0)})^{-1} \sum_{T_j\in\tubes_j}\sum_{H\in\mathcal{H}^{(0)}}\#\operatorname{CC}(T_j^H).
\end{equation}

\end{itemize}

At this point we will pause to check whether $\mathcal{H}^{(0)}$ and the sets $\{T_j^H\}_{H\in\mathcal{H}_0}$ satisfy the requirements of Lemma \ref{singleGrainDecompSecondStep}. 
\begin{itemize}
\item Property \ref{P1PmCompleteIntersection} is certainly true, since for each $(B_H,P_H,\mathcal{Q}_H)\in \mathcal{H}^{(0)},$ we have that $P_H = P_A$ for some $A\in\mathcal{A}$, and thus $(P_1,\ldots,P_m,P_H) = (P_1,\ldots,P_m,P_A)$ is a transverse complete intersection by Property \ref{C2Item} from Lemma \ref{singleGrainDecompFirstStep}.

\item Property \ref{areDisjoint} follows from the fact that the sets $\{\mathcal{Q}_A\}_{A\in\mathcal{A}}$ from Lemma \ref{singleGrainDecompFirstStep} are disjoint, plus the fact that the balls in $\mathcal{B}^\prime$ are disjoint. 

\item Property \ref{tubesContainedNeighborhoodGrains} follows from Property \ref{t1Item} of Lemma \ref{singleGrainDecompFirstStep}.

\item Property \ref{tubeSegmentLength} follows from the definition of $T_j^H$.

\item Property \ref{diamOfBall} follows from the definition of $\mathcal{B}.$
\end{itemize}

Our next task is to consider Properties \ref{numberOfTubeSegments} and \ref{smallYInEachGrain}. By the upper bound from \eqref{boundOnEachGrainSize} we have that for each $H\in\mathcal{H}^{(0)}$,
\begin{equation}\label{boundMassHInTermsA}
\begin{split}
\sum_{Q\in\mathcal{Q}_H} \Big(  \#\big\{
(Q, T_1,\ldots,T_k)&\in  Y^\prime \colon Q\in\mathcal{Q}_H,\ 
Q\cap T_j^H \neq\emptyset\big\}\Big)^{\frac{1}{k-1}}\\
& \lesssim_{\eps}  (\#\mathcal{A})^{-1+O(\eps/C)} (\#Y^\prime)^{\frac{1}{k-1}-O(\eps/C)} (\#\mathcal{Q})^{\frac{k-2}{k-1}},
\end{split}
\end{equation}
and thus \eqref{mostIncidencesStillCaptured} and \eqref{mostMassPreserved} imply that
\begin{equation}\label{lowerBoundCardH}
\#\mathcal{H}^{(0)}\gtrsim_\eps (\log M)^{-1} M^{-O(\eps/C)} (\#\mathcal{A}).
\end{equation}
While we could combine the terms $(\log M)^{-1}$ and $M^{-O(\eps/C)}$ into one, we will keep them distinct for now.
Combining \eqref{evenlySplitBetweenH0}, \eqref{boundMassHInTermsA}, and \eqref{lowerBoundCardH}, we conclude that for each $H\in\mathcal{H}^{(0)},$
\begin{equation}\label{massSpreadOverH0}
\begin{split}
\sum_{Q\in\mathcal{Q}_H} \Big(  \#\big\{
(Q, T_1,\ldots,T_k)&\in  Y^\prime \colon 
Q\cap T_j^H \neq\emptyset\big\}\Big)^{\frac{1}{k-1}}\\
&\lesssim_{\eps} (\log M)M^{O(\eps/C)}(\#\mathcal{H}^{(0)}) (\#Y^\prime)^{\frac{1}{k-1}} (\#\mathcal{Q})^{\frac{k-2}{k-1}}.
\end{split}
\end{equation}
The inequalities \eqref{mostMassPreserved}, \eqref{evenlySplitBetweenH0} and \eqref{massSpreadOverH0} imply that $\mathcal{H}^{(0)}$ satisfies Property \ref{smallYInEachGrain}; indeed, if $C$ is chosen sufficiently large then $\mathcal{H}^{(0)}$ satisfies \eqref{YGBddAboveBelow} with $M^{\eps/2}$ and $M^{-\eps/2}$ in place of $M^{\eps}$ and $M^{-\eps}$, respectively. This epsilon of slack will be useful for us in the arguments below.

By \eqref{sumOfSubTubes} and \eqref{lowerBoundCardH}, we have that for each index $j$ and each $T_j\in\tubes_j$,
\begin{equation*}
\begin{split}
\sum_{H\in\mathcal{H}^{(0)}}\#\operatorname{CC}(T_j^H)&\lesssim_{\eps} (\log M)M^{O(\eps/C)} (\# \mathcal{A})^{\frac{1}{n-m}}\\
&\lesssim_{\eps} (\log M)M^{O(\eps/C)}(\#\mathcal{H}^{(0)})^{\frac{1}{n-m}},
\end{split}
\end{equation*}
and thus for each index $j$ we have
\begin{equation}\label{numberCCTjH0}
\sum_{T_j\in\tubes_j}\sum_{H\in\mathcal{H}^{(0)}}\#\operatorname{CC}(T_j^H) \lesssim_{\eps} (\log M)M^{O(\eps/C)}(\#\tubes_j) (\# \mathcal{H}^{(0)})^{\frac{1}{n-m}}.
\end{equation}

Finally, by \ref{equalNumbersTubesEachCell}, this implies that for each index $j$ and each $H\in\mathcal{H}^{0}$, we have
\begin{equation}\label{boundNumberTubesEachCell}
\sum_{T_j\in\tubes_j}\#\operatorname{CC}(T_j^H) \lesssim_{\eps} (\log M)M^{O(\eps/C)}(\#\tubes_j) (\# \mathcal{H}^{(0)})^{\frac{1}{n-m}-1}.
\end{equation} 
Thus the set $\mathcal{H}^{(0)}$ and the sets $\{T_j^H\}_{H\in\mathcal{H}^{(0)}}$ satisfy all of the requirements of Lemma \ref{singleGrainDecompSecondStep}, except (crucially!) the first inequality in \eqref{numberTubeSegmentsTj} from Property \ref{numberOfTubeSegments}. To fix this, we will alternate between the following two steps. In the first step, we will cut the sub-tubes in $T_j^H$ into smaller sub-tubes so that Property \ref{numberOfTubeSegments} holds. Of course, when we cut the tubes in $T_j^H$ into smaller sub-tubes, they become shorter, and this this might cause Property \ref{diamOfBall} to fail. In the second step, we will cut the grains into smaller balls. This will cause Property \ref{diamOfBall} to hold, but now Property \ref{numberOfTubeSegments} might fail. We will iterate between these two steps multiple times, and eventually both Property \ref{diamOfBall} and \ref{numberOfTubeSegments} will hold simultaneously. 

Suppose that the lower bound from Property \ref{numberOfTubeSegments} fails for at least one index $j$. Then we perform the following step.

\medskip
\noindent{\bf Step 1.} For each index $j$ for which the lower bound of \eqref{numberTubeSegmentsTj} fails, cut each tube-segment in $\bigcup_{H\in\mathcal{H}^{(0)}}T_j^H$ into 
$$
X \sim  \delta^{-\eps}\frac{ (\#\tubes_j)(\#H^{(0)})^{\frac{1}{n-m}}}{\sum_{T\in\tubes_j}\sum_{H\in\mathcal{H}^{(0)}}\#\operatorname{CC}(T_j^H)}
$$
pieces of equal length. For each index $j$, let $\ell_j^{(1)}$ be the new lengths of the segments. If we choose the value of $X$ appropriately, then we can suppose that $\ell_j^{(1)}$ is of the form $\delta 2^N$ for some non-negative integer $N$.

After Step 1 has been performed, Property \ref{numberOfTubeSegments} now holds. Observe that each of Properties \ref{P1PmCompleteIntersection}--\ref{smallYInEachGrain} remain true, with the exception that Property \ref{diamOfBall} might fail. If Property \ref{diamOfBall} fails, then we form the following step. 
\medskip

\noindent{\bf Step 2.} Let $r^{(1)}=\max_j\ell_j^{(1)}$. For each $(B_H,P_H,\mathcal{Q}_H)\in\mathcal{H}^{(0)}$, cover $B_H$ by finitely overlapping balls of radius $2r^{(1)}$ so that each $r^{(1)}\times\delta$ tube is contained in at least one of these balls. For each such ball $B_H^\prime$, define $\mathcal{Q}_{H^\prime}$ to be the set of cubes from $\mathcal{Q}_H$ contained in $B_H^\prime$, and define $T_j^{H^\prime}$ to be the set of sub-tubes of $T_j^H$ contained in $B_H^\prime$. 

By dyadic pigeonholing, we can select a set $\mathcal{H}^{(1)}$, and for each index $j$; each $T_j\in\tubes_j$; and each $H=(B_H,P_H,\mathcal{Q}_H)\in\mathcal{H}^{(1)}$, a set $T_j^H$ so that $T_j^H$ is a disjoint union of sub-tubes of $T_j$ of length $\ell_j^{(1)}$ that are contained in $B_H$. We can make this selection so that \eqref{mostMassPreserved}, \eqref{evenlySplitBetweenH0}, and \eqref{equalNumbersTubesEachCell} hold with $\mathcal{H}^{(1)}$ in place of $\mathcal{H}^{(0)}$, except the term $(\log M)^{-1}$ in \eqref{mostMassPreserved} is replaced by $(\log M)^{-2}$. In particular, note that the implicit constant in the $O(\eps/C)$ terms in \eqref{massSpreadOverH0} and \eqref{numberCCTjH0} remain unchanged. 

If we repeat the arguments following \eqref{equalNumbersTubesEachCell}, we are now in the situation preceding Step 1. We iterate between these two steps until both Property \ref{diamOfBall} and \ref{numberOfTubeSegments} hold. Each iteration decreases at least one of the lengths $\ell_j$ by a multiplicative factor of $\delta^{\eps}$. Since Property \ref{diamOfBall} automatically holds if all sub-tubes have length $\leq\delta/2$, we conclude that we iterate the above procedure at most $k/\eps+1$ times.

Define $\mathcal{H}$ to be the set obtained by the final iteration. Each iteration weakens \eqref{massSpreadOverH0} and \eqref{numberCCTjH0} by a multiplicative factor of $\lesssim (\log M)^{-1}$, so all together these inequalities are weakened by a multiplicative factor of $\lesssim_{\eps}(\log M)^{-(k/\eps+1)}$. Since $(\log M)^{-(k/\eps+1)}\lesssim_{\eps} M^{\eps/2}$, this is an acceptable loss. Finally, choose the constant $C$ sufficiently large so that all terms of the form $O(\eps/C)$ are at most $\eps/2$.
\end{proof}

We are now ready to prove Proposition \ref{multiLevelGrains}. The result will be proved by repeated application of Lemma \ref{singleGrainDecompSecondStep}.

\begin{proof}[Proof of Proposition \ref{multiLevelGrains}]
We will construct the set $\mathcal{G}$ and the sets $\{T^G\}_{G\in\mathcal{G}}$ using the following iterative procedure. Define $\mathcal{G}^{(0)}=\{(B(0,1),P_0)\}$, where $P_0$ is the zero polynomial. For each index $j$, define $\ell_{0,j}=1$ and for each $T_j\in\tubes_j$, define $T_j^G = T$. Define $D_0 = 1$. Define $Y_{G_0}=Y$.

Observe that $\mathcal{G}^{(0)}$ is a tree of grains of depth 0 and complexity $\lesssim_{\eps}1$, and for each index $j$ and each $T_j\in\tubes_j$, we have that $\{T_j^G\}_{G\in\mathcal{G}^{(0)}}$ is compatible with $\mathcal{G}^{(0)}$. The set $\mathcal{G}^{(0)}$ and the set systems $\{T_j^G\}_{G\in\mathcal{G}^{(0)}}$ satisfy all the properties from Proposition \ref{multiLevelGrains} that apply to grains in $\mathcal{G}$ that have level $i=0$. 

In fact, the following slightly stronger variant of \ref{tubesTouchFewGrainsEquation} is true for all grains $G\in\mathcal{G}^{(0)}$ that have level $i=0$: 
\begin{equation}\label{strongerTubesTouchFewGrainsEquation}
M^{-i \eps/m} (D_1^{1-n}\cdots D_i^{i-n})(\#\tubes_j)\lesssim_{\eps} \sum_{T_j\in\tubes_j} \#\operatorname{CC}(T_j^G)\lesssim_{\eps} M^{i \eps/m}(D_1^{1-n}\cdots D_i^{i-n})(\#\tubes_j).
\end{equation}
This inequality is slightly silly, since both the terms $M^{-i \eps/m}$ and $(D_1^{1-n}\cdots D_i^{i-n})$ are equal to 1. However, stating the inequality in this way will be useful for us later on.

Similarly, the following variant of \eqref{YGBddAbove} is true for all grains $G\in\mathcal{G}^{(0)}$ that have level $i=0$:
\begin{equation}\label{strongerYGBddAbove}
\begin{split}
&M^{-i \eps/m} (D_1^{-n}\cdots D_i^{i-n-1})
\sum_{Q} \big(\#\{
(Q, T_1,\ldots,T_k)\in Y\}\big)^{{\frac{1}{k-1}}}\\
&\qquad \sum_{Q} \big(  \#\{
(Q, T_1,\ldots,T_k)\in  Y_{G} \colon 
Q\cap T_j^G \neq\emptyset\ \textrm{for each index}\ j\}\big)^{\frac{1}{k-1}}
\\
&\qquad\qquad\lesssim_{\eps} \ M^{i\eps/m} (D_1^{-n}\cdots D_i^{i-n-1})
\sum_{Q} \big(\#\{
(Q, T_1,\ldots,T_k)\in Y\}\big)^{{\frac{1}{k-1}}}.
\end{split}
\end{equation}
Again, this inequality is slightly silly since both the terms $M^{i \eps/m}$ and $(D_1^{-n}\cdots D_i^{k-n-1})$ are equal to 1. 

Suppose we have constructed a tree of grains $\mathcal{G}^{(i)}$ of depth $i$; numbers $\{ \ell_{i^\prime,j}\}$ for $0\leq i^\prime\leq i$ and $1\leq j\leq k$; set systems $\{T_j^G\}_{G\in\mathcal{G}^{(i)}}$; and sets $\{Y_G\}_{G\in\mathcal{G}^{(i)}}$ that satisfy the requirements of Proposition \ref{multiLevelGrains}, as well as the stronger inequalities \eqref{strongerTubesTouchFewGrainsEquation} and \eqref{strongerYGBddAbove}. We will show how to append additional leaves to this tree to construct level $i+1$.

Let $G=(B,P_1,\ldots,P_i)\in\mathcal{G}^{(i)}$ be a grain that has level $i$. For each index $j$, define 
$$
\tubes_{G,j}=\bigcup_{T_j\in\tubes_j}\operatorname{CC}(T^G).
$$ 
Thus $\tubes_{G,j}$ is a set of $\ell_{i,j}\times\delta$ tubes, each of which is contained in the ball $B$, which has radius at most $4\max_j \ell_{i,j}+\delta$. By \eqref{strongerTubesTouchFewGrainsEquation}, we have
\begin{equation}\label{cardOfTubesGj}
M^{-i \eps/m} (D_1^{1-n}\cdots D_i^{1-i})(\#\tubes_j)\lesssim_{\eps} \#\tubes_{G,j}\lesssim_{\eps} M^{i \eps/m}(D_1^{1-n}\cdots D_i^{1-i})(\#\tubes_j).
\end{equation}

Apply Lemma \ref{singleGrainDecompSecondStep} to $B$, $P_1,\ldots,P_i$, $\tubes_{G,1},\ldots,\tubes_{G,j}$, and $Y_G$, with allowable error $\eps/C$. $C$ is a constant that will be chosen below. We obtain a set $\mathcal{H}_G$; lengths $\ell_{G,1},\ldots,\ell_{G,k}$; and for each index $j$ and each $T_j\in\tubes_{G,j}$, families of sets $\{T_j^H\}_{H\in\mathcal{H}_G}$.

Observe that $\#\mathcal{H}_G$ is an integer between $0$ and $M$, and each of $\ell_{G,1},\ldots,\ell_{G,k}$ are numbers of the form $\delta 2^N,$ where $N$ is an integer between $0$ and $\log(1/\delta)$. Thus by dyadic pigeonholing, there is a number $L$ and lengths $\ell_{i+1,1},\ell_{i+1,2},\ldots,\ell_{i+1,k}$ so that if we define $\mathcal{G}^{(i)}_*$ to be the set of grains $G\in\mathcal{G}^{(i)}$ that have level $i$ and for which $L\leq \#\mathcal{H}_G\leq 2L$ and $\ell_{G,j}=\ell_{i+1,j}$ for each index $j$, then
\begin{equation}
\#\mathcal{G}^{(i)}_*\gtrsim (\log M)^{-k-1} D_1^n\cdots D_i^{n-i+1}.
\end{equation}
Define 
\begin{equation}
\mathcal{G}^{(i+1)} = \mathcal{G}^{(i)}\cup\ \bigcup_{G=(B,P_1,\ldots,P_i)\in\mathcal{G}^{(i)}_*}\{ (B_H,P_1,\ldots,P_i,P_H)\colon H\in\mathcal{H}_G\}.
\end{equation}
For each of these newly defined grains $G=(B_H,P_1,\ldots,P_i,P_H)$, for each index $j$, and for each $T_j\in\tubes_j$, define 
$$
T_j^G=  \bigcup_{\tilde T_j\in \operatorname{CC}(T_j^{G^\prime})}\tilde T_j^{H},
$$
where $G^\prime=G|_{i}\in\mathcal{G}^{(i)}$. By construction, $T_j^G$ is a disjoint union of $\ell_{i+1,j}\times\delta$ tubes, each of which is contained in $N_{3\sqrt n\delta}(Z(P_1,\ldots,P_i,P_H))$. 

Define $D_{i+1}$ so that
$$
D_1^n\cdots D_i^{n-i+1}D_{i+1}^{n-i}=\#\{ G\in\mathcal{G}\colon \operatorname{level}(G) = i+1 \}.
$$
We have $D_1^n\cdots D_{i+1}^{n-i}\sim L \# \mathcal{G}^{(i)}_*$, and since 
$$
(\log M)^{-k-1}D_1^n\cdots D_i^{n-i+1}\lesssim  \# \mathcal{G}^{(i)}_*\leq D_1^n\cdots D_i^{n-i+1},
$$
we have
\begin{equation}\label{boundOnSizeDip1}
(\log M)^{-k-1}L\lesssim D_{i+1}^{n-i}\lesssim L.
\end{equation}

We can now verify that if the constant $C$ is chosen sufficiently large, then the tree $\mathcal{G}^{(i+1)}$ has the following properties.
\begin{itemize}
\item First, $\mathcal{G}^{(i+1)}$ is a tree of grains of depth $i+1$ and complexity $\lesssim_{\eps}1$, since for each $G=(B,P_1,\ldots,P_{i+1})\in\mathcal{G}^{i+1}$, we have that $Z(P_1,\ldots,P_{i+1})$ is a transverse complete intersection.

\item Second, for each index $j$ and each $T_j\in\tubes_j$, the set system $\{T_j^G\}_{G\in\mathcal{G}^{(i+1)}}$ is compatible with $\mathcal{G}$.

\item Properties \ref{tubesUniformLength} and \ref{grainsLocalized} follow from Properties \ref{tubeSegmentLength} and \ref{tubesContainedNeighborhoodGrains} from Lemma \ref{singleGrainDecompSecondStep}, respectively. 

\item Property \ref{tubesTouchFewGrains} at level $i+1$ follows from the stronger variant \eqref{strongerTubesTouchFewGrainsEquation} at level $i+1$, which in turn follows from \eqref{strongerTubesTouchFewGrainsEquation} at level $i$, \eqref{cardOfTubesGj}, \eqref{boundOnSizeDip1}, and Property \ref{numberOfTubeSegments} from Lemma \ref{singleGrainDecompSecondStep}. 

\item Property \ref{mostOfYCaptured} at level $i+1$ follows from the stronger variant \eqref{strongerYGBddAbove} at level $i+1$, which in turn follows from \eqref{strongerYGBddAbove} at level $i$, \eqref{boundOnSizeDip1}, and Property \ref{smallYInEachGrain} from Lemma \ref{singleGrainDecompSecondStep}.
\end{itemize}
To complete the proof, define $\mathcal{G} = \mathcal{G}^{(m)}.$
\end{proof}

\section{$k$-linear Kakeya estimates for direction-separated tubes}\label{multilinearKakeyaSec}
In this section we will prove Theorem \ref{kLinearKakeyaThm}. We will actually prove the following slightly more technical version. 
\begin{kLinearKakeyaThmRestate}
Let $2\leq k\leq n$ and let $\tubes_1,\ldots,\tubes_k$ be sets of direction-separated $1\times\delta$ tubes in $\RR^n$, each of which have cardinality at most $N$. Then for each $\eps>0$, there is a constant $C(n,\eps)$ so that
\begin{equation}\label{kLinRdWithN}
\Big\Vert\Big(\sum_{T_1\in\tubes_1}\cdots\sum_{T_k\in\tubes_k} \chi_{T_1}\cdots\chi_{T_k}|v_1\wedge\ldots\wedge v_k|^{k/d}\Big)^{\frac{1}{k}} \Big\Vert_{\frac{d}{d-1}} \leq C(n,\eps) \Big(\frac{1}{\delta}\Big)^{\frac{n}{d}-1+\eps} (\delta^{n-1}N)^{\frac{n(d-1)}{d(n-1)}},
\end{equation}
where
\begin{equation}\label{valueOfDmultiLinKakeyaWithN}
d = \frac{n^2+k^2+n-k}{2n}.
\end{equation}
\end{kLinearKakeyaThmRestate}

\begin{proof}
Our first step will be to discretize the problem. We have 
\begin{equation}\label{sumOverCubes}
\begin{split}
\int_{\RR^d}& \Big(\sum_{T_1\in\tubes_1}\cdots\sum_{T_k\in\tubes_k} \chi_{T_1}\cdots\chi_{T_k}|v_1\wedge\ldots\wedge v_k|^{\frac{k}{d}}\Big)^{\frac{d}{k(d-1)}}\\
&\leq\sum_{\substack{Q\subset \RR^n\\ Q\ \textrm{a}\ \delta\textrm{-cube}}} |Q| \sum_{\substack{\delta\leq\theta\lesssim 1\\ \theta\ \textrm{dyadic}}} \theta^{\frac{1}{d-1}} \Big(\#\big\{(T_1,\ldots,T_k)\in\tubes_1(Q)\times\cdots\times\tubes_k(Q)\colon\\
&\qquad\qquad\qquad\qquad\qquad\qquad\qquad\qquad\qquad \theta<|v(T_1)\wedge\ldots\wedge v(T_k)|\leq 2\theta \big\}\Big)^{\frac{d}{k(d-1)}}.
\end{split}
\end{equation}
By dyadic pigeonholing, there exist numbers $\delta\leq\theta\leq 1$ and $1\leq\mu\leq \delta^{1-n}$; a set $K\subset\RR^n$ that is a union of $\delta$-cubes; and numbers $\mu_1,\ldots,\mu_k$ so that
\begin{equation}\label{defnOfMu}
\#\big\{(T_1,\ldots,T_k)\in \tubes_1(Q)\times\cdots\times\tubes_k(Q) \colon \theta<|v_1\wedge\ldots\wedge v_k|\leq 2\theta\big\}\sim\mu^k
\end{equation}
for every $\delta$-cube $Q\subset K$;
\begin{equation}\label{RHSBound}
\textrm{RHS}\ \eqref{sumOverCubes}\lessapprox \theta^{1/(d-1)}  \mu^{\frac{d}{d-1}}|K|;
\end{equation}
and
\begin{equation}\label{defnMuj}
\#\tubes_j(Q)\sim \mu_j\quad\textrm{for every $\delta$-cube}\ Q\subset K\ \textrm{and every index}\ j.
\end{equation}
Note that \eqref{defnOfMu} and \eqref{defnMuj} imply that
\begin{equation}\label{muBddMuAvg}
\mu^k \leq \mu_1\cdots\mu_k.
\end{equation} 

Define 
$$
Y = \big\{(Q,T_1,\ldots,T_k)\colon Q\subset K,\ \theta< |v_1\wedge\ldots\wedge v_k|\leq 2\theta,\ T_j\in\tubes_j(Q)\ \textrm{for each index}\ j \big\}.
$$
By \eqref{defnOfMu} we have
\begin{equation}\label{KThetaMuComparedToY}
|K|\theta^{\frac{1}{k-1}}\mu^{\frac{k}{k-1}}\sim\sum_{Q\subset K}|Q|\theta^{\frac{1}{k-1}}\big(\#\{(Q,T_1,\ldots,T_k)\in Y\}\big)^{\frac{1}{k-1}}.
\end{equation}

Apply Proposition \ref{multiLevelGrains} to $\tubes_1,\ldots,\tubes_k$ and $Y$, with $m=n-k$ and allowable error $\eps/C$; here $C$ is a large constant to be chosen later. We obtain a tree $\mathcal{G}$ of grains of depth $n-k$ and complexity $E\lesssim_{\eps}1$; lengths $\ell_{i,j}$; and for each index $j$ and each $T_j\in\tubes_j$, a set system $\{T_j^G\}_{G\in\mathcal{G}}$ of sub-tubes of $T_j$ that is compatible with $\mathcal{G}$. 

Observe that since the tubes in $\tubes$ are direction separated, we have $\#\tubes\lesssim \delta^{1-n}$ and thus the quantity $M$ from \eqref{defnM} satisfies $M\lesssim \delta^{-1-k(n-1)}$. By \eqref{KThetaMuComparedToY} and Property \ref{mostOfYCaptured} of Proposition \ref{multiLevelGrains}, we have that for each $G\in\mathcal{G}$ that has level $n-k$, 
\begin{equation}\label{controllingKThetaMu}
\begin{split}
|K|&\theta^{{\frac{1}{k-1}}}\mu^{\frac{k}{k-1}} \\
& \lesssim_{\eps} \delta^{-O(\eps/C)}(D_1^n\cdots D_{n-k}^{k+1}) \sum_{Q}|Q| \theta^{\frac{1}{k-1}} \big(\#\{(Q,T_1,\ldots,T_k)\in Y\colon T_j^G\cap Q\neq\emptyset\ \textrm{for each index}\ j\}\big)^{\frac{1}{k-1}}\\
&\lesssim \delta^{-O(\eps/C)}(D_1^n\cdots D_{n-k}^{k+1}) \int \bigg( \sum_{\substack{T_1\in\tubes_1\\ T_1^G\neq\emptyset}}\ldots \sum_{\substack{T_k\in\tubes_k\\ T_k^G\neq\emptyset}}\chi_{T_1}\cdots\chi_{T_k}|v_1\wedge\ldots\wedge v_k|\bigg)^{\frac{1}{k-1}}.
\end{split}
\end{equation}

Our next task is to obtain a lower bound for the numbers $D_1,\ldots,D_{n-k}$ from Property \ref{tubesTouchFewGrains} of Proposition \ref{multiLevelGrains}. Fix an index $1\leq i\leq n-k$ and an index $j$. Recall that for each $T_j\in\tubes_j$, the sets $\{T^G\colon G\in\mathcal{G},\ \operatorname{level}(G) = i\}$ are disjoint, and each of these sets is a disjoint union of $\ell_{i,j}\times\delta$ sub-tubes of $T_j$. Similarly, the sets $\{T^G\colon G\in\mathcal{G},\ \operatorname{level}(G) = i-1\}$ are disjoint, and each of these sets is a disjoint union of $\ell_{i-1,j}\times\delta$ sub-tubes of $T_j$. Thus
\begin{equation}\label{ineqSuccessiveLevels}
\ell_{i,j}\sum_{T_j\in\tubes_j}\sum_{ \substack{G\in\mathcal{G}\\ \operatorname{level}(G) = i}}  \#\operatorname{CC}(T_j^G) \leq
\ell_{i-1,j}\sum_{T_j\in\tubes_j}\sum_{ \substack{G\in\mathcal{G}\\ \operatorname{level}(G) = i-1}}  \#\operatorname{CC}(T_j^G).
\end{equation}
But by Property \ref{tubesTouchFewGrains} of Proposition \ref{multiLevelGrains} and the definition of $D_1,\ldots,D_i$, we have that 
\begin{equation}\label{numberTubeSegmentsEachLevel}
\sum_{T_j\in\tubes_j}\sum_{\substack{G\in\mathcal{G}\\ \operatorname{level}(G) = i}}  \#\operatorname{CC}(T_j^G)  
 \gtrsim_\eps M^{-O(\eps/C)}D_{i}\sum_{T_j\in\tubes_j}\sum_{ \substack{G\in\mathcal{G}\\ \operatorname{level}(G) = i-1}} \#\operatorname{CC}(T_j^G).
\end{equation}
Combining \eqref{ineqSuccessiveLevels} and \eqref{numberTubeSegmentsEachLevel}, we conclude that
\begin{equation}\label{growthOfLi}
\ell_{i-1,j}\gtrsim_{\eps}M^{-O(\eps/C)} D_i \ell_{i,j}.
\end{equation} 

We will make use of \eqref{growthOfLi} as follows. Let $G=(B,P_1,\ldots,P_i)\in\mathcal{G}$ be a grain. Let $j_i$ be an index so that $\ell_{i,j_i}$ is maximal. By Property \ref{grainsLocalized} of Proposition \ref{multiLevelGrains}, we have that $B$ has radius $\leq 4\ell_{i,j_i}+\delta$. In particular, for each tube $T_{j_i}\in\tubes_{j_i}$, we have that $T_{j_i}^G$ is either empty, or is a union of at most four disjoint $\ell_{i,{j_i}}\times\delta$ tubes. If $T_{j_i}^G$ is non-empty, then for each index $1\leq i^\prime<i$ we have that each of the $\ell_{i,{j_i}}\times\delta$ tubes in $T_{j_i}^G$ is contained in a $\ell_{i^\prime,j_i}\times\delta$ sub-tube from $T_{j_i}^{G|_{i^\prime}}$. By \eqref{growthOfLi} we have 
$$
\ell_{i^\prime,j_i} \gtrsim_{\eps}M^{-O(\eps/C)} D_{i^\prime+1}\cdots D_{i} \ell_{i,j_i},
$$
and thus if we define $B_{i^\prime}$ to be the ball with the same center as $B$ and radius $M^{-O(\eps/C)}D_{i^\prime+1}\cdots D_{i} \ell_{i,j_i}$, then 
\begin{equation}\label{flagOfVarieties}
|T_{j_i}\cap B_{i^\prime}\cap N_{3\sqrt n\delta}\big(Z(P_1,\ldots,P_{i^\prime})\big)|\gtrsim M^{O(\eps/C)}D_{i^\prime+1}\cdots D_{i} \ell_{i,j_i}\delta^{n-1}\quad\textrm{for each index}\ 1\leq i^\prime\leq i.
\end{equation}

Applying Theorem \ref{polyWolffFlagsVarities} with allowable error $\eps/C$, we obtain
\begin{equation}\label{upperBoundTubesSiG}
\begin{split}
\sum_{T_{j_i}\in\tubes_{j_i}}\operatorname{CC}(T_{j_i}^G)&\leq 4\ \#\{T\in\tubes_{j_i}\colon T^G\neq\emptyset\} \\
&\lesssim_{\eps} \frac{\delta^{1+i-n-\eps/C}}{\prod_{i^\prime=1}^{i} (\ell_{i,j_i}D_i\cdots D_{i^\prime+1})}\\
&\sim \frac{ \delta^{1+i-n-\eps/C}}{ \ell_{i,j_i}^{i} D_2D_3^2\cdots D_{i}^{i-1}}.
\end{split}
\end{equation}
Comparing the upper bound \eqref{upperBoundTubesSiG} with the lower bound from Property \ref{tubesTouchFewGrains} of Proposition \ref{multiLevelGrains}, we see that

\begin{equation}\label{compareLowerBoundTubesSiUpperBoundTubesSi}
D_1^{1-n}\cdots D_i^{i-n}(\#\tubes_{j_i}) \lesssim_\eps \frac{ \delta^{1+i-n-O(\eps/C)}}{ \ell_{i,j_i}^{i} D_2D_3^2\cdots D_{i}^{i-1}}.
\end{equation}
This inequality should be thought of as a lower bound for the numbers $D_1,\ldots,D_i$. Note that for each index $1\leq j\leq k$, we have 

\begin{equation}
\begin{split}
|K|\mu_j & \lesssim \sum_{T\in\tubes_j}|T\cap K|\\
&\leq \sum_{T\in\tubes_j}\sum_{\substack{G\in\mathcal{G}\\ \operatorname{level}(G)=i}} |T^G\cap K|\\
& \lesssim \delta^{-O(\eps/C)}(D_1\cdots D_i)\ell_{i,j} (\delta^{n-1}\#\tubes_j),
\end{split}
\end{equation}
where for the final inequality we used Properties \ref{tubesUniformLength} and \ref{tubesTouchFewGrains} (the latter summed over all $G\in\mathcal{G}$ with $\operatorname{level}(G)=i$) of Proposition \ref{multiLevelGrains}. 
We conclude that 
\begin{equation}
\ell_{i,j}\gtrsim_{\eps} \frac{ \delta^{O(\eps/C)}|K|\mu_j }{(\delta^{n-1}\#\tubes_j) D_1\cdots D_i}=\frac{\delta^{O(\eps/C)}\lambda_j}{D_1\cdots D_i},
\end{equation}
where we define 
\begin{equation}\label{defnLambdaJ}
\lambda_j = |K|\mu_j (\delta^{n-1}\#\tubes_j)^{-1}.
\end{equation}
In particular, if we define $\lambdabar=(\lambda_1\cdots\lambda_k)^{1/k}$, then since $\ell_{i,j_i}=\max_j\ell_{i,j}$, we have
\begin{equation}\label{boundOnEllSii}
\ell_{i,j_i}\geq \frac{\delta^{O(\eps/C)}\lambdabar}{
D_1\cdots D_i}.
\end{equation}
Combining \eqref{boundOnEllSii} and \eqref{compareLowerBoundTubesSiUpperBoundTubesSi}, we have
$$
D_1^{1-n}\cdots D_i^{i-n}(\#\tubes_{j_i}) \lesssim  \lambdabar^{-i} \delta^{1+i-n-O(\eps/C)} D_1^{i}D_2^{i-1}\cdots D_{i}^1,
$$
or
\begin{equation}\label{lowerBoundProductDTubesJi}
\delta^{-i+O(\eps/C)}\lambdabar^i (\delta^{n-1}\#\tubes_{j_i}) \lesssim D_1^{n+i-1}D_2^{n+i-3}D_3^{n+i-5}\cdots D_i^{n-i+1}.
\end{equation}
Define 
$$
W = \min_{1\leq j\leq k}(\#\tubes_j).
$$ 
Then \eqref{lowerBoundProductDTubesJi} implies

\begin{equation}\label{lowerBoundProdDjLeveli}
\begin{split}
\delta^{-i+O(\eps/C)}\lambdabar^i (\delta^{n-1}W ) & \lesssim D_1^{n+i-1}D_2^{n+i-3}D_3^{n+i-5}\cdots D_i^{n-i+1}\\
&=\prod_{m=1}^i D_m^{n+j+1-2m}.
\end{split}
\end{equation}

We will need to average certain powers of \eqref{lowerBoundProdDjLeveli} as $i$ ranges from $1$ to $n-k$. A computation shows that
\begin{equation}\label{keyComp}
\prod_{i=1}^{n-k}\Big(\prod_{m=1}^i D_m^{n+i+1-2m}\Big)^{\frac{k(k-1)}{(n - i+1)(n - j)(n - i-1)}} = D_1^{n-k}D_2^{n-k-1}\cdots D_{n-k+1}^2D_{n-k}.
\end{equation}
(Recall that $k\geq 2$, so the denominator $(n - j+1)(n - j)(n - j-1)$ is never $0$). Combining \eqref{lowerBoundProdDjLeveli} and \eqref{keyComp}, we obtain
\begin{equation}\label{boundOnProdDj}
\begin{split}
D_1^{n-k}\cdots D_{n-k}&\geq \prod_{i=1}^{n-k}\Big( \delta^{-i+O(\eps/C)}\lambdabar^i (\delta^{n-1}W ) \Big)^{\frac{k(k-1)}{(n - i+1)(n - i)(n - i-1)}}\\
& = (\lambdabar/\delta)^{\frac{n^2+k^2-2kn+n-k}{2n}}(\delta^{n-1}W)^{\frac{1}{2} -\frac{k(k-1)}{2n(n-1)}}.
\end{split}
\end{equation}
Observe that since $2\leq k\leq n$, the term $(\delta^{n-1}W)$ is raised to an exponent $0\leq\alpha<1/2$. 

Next, let $G\in\mathcal{G}$ be a grain of level $n-k$. 
%
By \eqref{controllingKThetaMu} and Theorem \ref{multilinearKakeyaThm} we have 
\begin{equation*}
\begin{split}
\theta^{1/(k-1)}\mu^{k/(k-1)}|K| & \lesssim_{\eps} \delta^{-O(\eps/C)}(D_1^n\cdots D_{n-k}^{k+1}) \int \bigg( \sum_{\substack{T_1\in\tubes_1\\ T_1^G\neq\emptyset}}\ldots \sum_{\substack{T_k\in\tubes_k\\ T_k^G\neq\emptyset}}\chi_{T_1}\cdots\chi_{T_k}|v_1\wedge\ldots\wedge v_k|\bigg)^{1/(k-1)}\\
&\lesssim_{\eps} (D_1^n\cdots D_{n-k}^{k+1})\Big(\frac{1}{\delta}\Big)^{\frac{n-k}{k-1}+O(\eps/C)}\prod_{j=1}^k\Big((\delta^{n-1}\#\tubes_j) D_1^{1-n}D_2^{2-n}\cdots D_{n-k}^{-k}\Big)^{1/(k-1)},
\end{split}
\end{equation*}
where on the second line we used Property \ref{tubesTouchFewGrains} from Proposition \ref{multiLevelGrains}. This implies that
\begin{equation}\label{boundPowersThetaMuK}
\theta\mu^k|K|^{k-1}(D_1^n\cdots D_{n-k}^{k+1})^{-(k-1)}\lesssim_{\eps} \delta^{k-n-O(\eps/C)}\prod_{j=1}^k\Big((\delta^{n-1}\#\tubes_j) D_1^{1-n}D_2^{2-n}\cdots D_{n-k}^{-k}\Big).
\end{equation}
Re-arranging, we obtain
\begin{equation}
\begin{split}
\theta\mu^k|K|^{k-1}   &\lesssim_{\eps}  \delta^{k-n-O(\eps/C)}\Big(\prod_{j=1}^k(\delta^{n-1}\#\tubes_j)\Big) \big(D_1^{n}D_2^{n-1}\cdots D_{n-k}^{k+1}\big)^{k-1} \big(D_1^{1-n}D_2^{2-n}\cdots D_{n-k}^{-k}\big)^{k}\\
&\leq \delta^{k-n-O(\eps/C)}\Big((\delta^{n-1}W)^{1/2}(\delta^{n-1}N)^{k-1/2}\Big)\Big(D_1^{n-k}D_2^{n-k-1}\cdots D_{n-k}\Big)^{-1}.
\end{split}
\end{equation} 
Inserting \eqref{boundOnProdDj} and recalling that $W\leq N$, we conclude
\begin{equation}\label{yetAnotherBoundThetaMuK}
\begin{split}
\theta&\mu^k|K|^{k-1} \\
&\lesssim_{\eps} \delta^{k-n-O(\eps/C)}\Big((\delta^{n-1}W)^{1/2}(\delta^{n-1}N)^{k-1/2}\Big)
\Big((\lambdabar/\delta)^{\frac{n^2+k^2-2kn+n-k}{2n}}(\delta^{n-1}W)^{\frac{1}{2} -\frac{k(k-1)}{2n(n-1)}}\Big)^{-1}\\
%
&\lesssim_{\eps} \delta^{k-n+\frac{n^2+k^2-2kn+n-k}{2n}-O(\eps/C)} \lambdabar^{-\frac{n^2+k^2-2kn+n-k}{2n}}(\delta^{n-1}N)^{k-1/2+\frac{k(k-1)}{2n(n-1)} }.
\end{split}
\end{equation}

From \eqref{muBddMuAvg} and \eqref{defnLambdaJ}, we have
$$
\lambdabar = |K| \Big(\prod_{j=1}^k \mu_j (\delta^{n-1}\#\tubes_j)^{-1} \Big)^{1/k}\geq |K|\mu (\delta^{n-1}N)^{-1},
$$
and thus \eqref{yetAnotherBoundThetaMuK} implies
\begin{equation}
\theta^{2n}\mu^{n^2+k^2+n-k}|K|^{n^2+k^2-n-k}  \lesssim_{\eps} \delta^{-n^2+k^2+n-k-O(\eps/C)} (\delta^{n-1}N)^{n^2+k^2-k+\frac{k(k-1)}{n-1} }.
\end{equation}
Recalling the definition of $d$ from \eqref{valueOfDmultiLinKakeyaWithN}, we have
\begin{equation}\label{boundOnThetaMuK}
\begin{split}
\theta^{\frac{1}{d-1}}  \mu^{\frac{d}{d-1}}|K| & = \theta^{\frac{2n}{n^2+k^2-n-k}}\mu^{\frac{n^2+k^2+n-k}{n^2+k^2-n-k}} |K|  \\
& \lesssim_{\eps} \delta^{\frac{-n^2+k^2+n-k}{n^2+k^2-n-k}-O(\eps/C)} (\delta^{n-1}N)^{\frac{n}{n-1}}\\
&=\Big(\frac{1}{\delta}\Big)^{\frac{n-d}{d-1}+O(\eps/C)} (\delta^{n-1}N)^{\frac{n}{n-1}}.
\end{split}
\end{equation}
Combining \eqref{sumOverCubes}, \eqref{RHSBound}, and \eqref{boundOnThetaMuK}, and selecting $C$ sufficiently large, we obtain \eqref{kLinRdWithN}.
\end{proof}

\bibliographystyle{abbrv}
\bibliography{KakeyaRd}

\begin{thebibliography}{10}

\bibitem{BPR}
S.~Basu, R.~Pollack, and M.-F. Roy.
\newblock On the number of cells defined by a family of polynomials on a
  variety.
\newblock {\em Mathematika}, 43:120--126, 1996.

\bibitem{BCT}
J.~Bennett, A.~Carbery, and T.~Tao.
\newblock On the multilinear restriction and {K}akeya conjectures.
\newblock {\em Acta Math.}, 196:261--302, 2006.

\bibitem{Bes}
A.~Besicovitch.
\newblock Sur deux questions d'integrabilite des fonctions.
\newblock {\em J. Soc. Phys. Math.}, 2:105--123, 1919.

\bibitem{BCR}
J.~Bochnak, M.~Coste, and M.-F. Roy.
\newblock {\em Real algebraic geometry}.
\newblock Springer-Verlag, 1998.

\bibitem{BG}
J.~Bourgain and L.~Guth.
\newblock Bounds on oscillatory integral operators based on multilinear
  estimates.
\newblock {\em Geom. Funct. Anal.}, 21:1239--1295, 2011.

\bibitem{Burguet}
D.~Burguet.
\newblock A proof of {Y}omdin-{G}romov's algebraic lemma.
\newblock {\em Israel J. Math.}, 168:291--316, 2008.

\bibitem{CV}
A.~Carbery and S.~I. Valdimarsson.
\newblock The endpoint multilinear {K}akeya theorem via the {B}orsuk-{U}lam
  theorem.
\newblock {\em J. Funt. Anal.}, 264:1643--1663, 2012.

\bibitem{Cor}
A.~Cordoba.
\newblock The {K}akeya maximal function and the spherical summation
  multipliers.
\newblock {\em Am. J. Math.}, 99:1--22, 1977.

\bibitem{G}
L.~Guth.
\newblock The endpoint case of the {B}ennett-{C}arbery-{T}ao multilinear
  {K}akeya conjecture.
\newblock {\em Acta Math.}, 205:263--286, 2010.

\bibitem{G2}
L.~Guth.
\newblock A restriction estimate using polynomial partitioning.
\newblock {\em J. Amer. Math. Soc.}, 29:371--413, 2016.

\bibitem{G3}
L.~Guth.
\newblock Restriction estimates using polynomial partitioning {I}{I}.
\newblock {\em Acta Math.}, 221:81--142, 2018.

\bibitem{GK}
L.~Guth and N.~Katz.
\newblock On the {E}rd{\H o}s distinct distances problem in the plane.
\newblock {\em Ann. of Math.}, 181:155--190, 2015.

\bibitem{GZ}
L.~Guth and J.~Zahl.
\newblock Polynomial {W}olff axioms and {K}akeya-type estimates in
  $\mathbb{R}^4$.
\newblock {\em Proc. London Math. Soc.}, 117:192--220, 2018.

\bibitem{HR1}
J.~Hickman and K.~M. {Rogers}.
\newblock {Improved Fourier restriction estimates in higher dimensions}.
\newblock {\em Camb. J. Math.}, 7:219--282, 2019.

\bibitem{HR2}
J.~{Hickman} and K.~M. {Rogers}.
\newblock {New Kakeya estimates using the polynomial Wolff axioms}.
\newblock {\em arXiv e-prints}, (arXiv:1901.01802), 2019.

\bibitem{HZ}
J.~Hickman and J.~Zahl.
\newblock A note of {F}ourier restriction and nested polynomial {W}olff axioms.
\newblock {\em arXiv e-prints}, (arXiv:2010.02251), 2020.

\bibitem{KR}
N.~Katz and K.~Rogers.
\newblock On the polynomial {W}olff axioms.
\newblock {\em Geom. Funct. Anal.}, 28:1706--1716, 2018.

\bibitem{KT}
N.~Katz and T.~Tao.
\newblock Recent progress on the {K}akeya conjecture.
\newblock {\em Publ. Mat.}, 46:161--179, 2002.

\bibitem{KZ}
N.~Katz and J.~Zahl.
\newblock An improved bound on the {H}ausdorff dimension of {B}esicovitch sets
  in $\mathbb{R}^3$.
\newblock {\em J. Amer. Math. Soc.}, 32:195--259, 2019.

\bibitem{KZ2}
N.~Katz and J.~Zahl.
\newblock A {K}akeya maximal function estimate in four dimensions using
  planebrushes.
\newblock {\em to appear, Rev. Mat. Iberoam.}, (also in arXiv:1902.00989),
  2019.

\bibitem{Kuttler2}
K.~Kuttler.
\newblock {\em Modern Analysis}.
\newblock CRC Press, 1998.

\bibitem{Kuttler}
K.~Kuttler.
\newblock {\em Real and Abstract Analysis}.
\newblock https://math.byu.edu/~klkuttle/541book.pdf, 2019.

\bibitem{Mil}
J.~Milnor.
\newblock On the {B}etti numbers of real varieties.
\newblock {\em Proc. Amer. Math. Soc.}, 15:275--280, 1964.

\bibitem{W1}
T.~Wolff.
\newblock An improved bound for {K}akeya type maximal functions.
\newblock {\em Rev. Mat. Iberoam.}, 11:651--674, 1995.

\bibitem{W2}
T.~Wolff.
\newblock Recent work connected with the {K}akeya problem.
\newblock In {\em Prospects in mathematics (Princeton, NJ)}, pages 129--162,
  1996.

\bibitem{Wongkew}
Wongkew.
\newblock Volumes of tubular neighbourhoods of real algebraic varieties.
\newblock {\em Pacific J. Math.}, 159:177--184, 1993.

\bibitem{Z}
J.~Zahl.
\newblock A discretized {S}everi-type theorem with applications to harmonic
  analysis.
\newblock {\em Geom. Funct. Anal.}, 28:1131--1181, 2018.

\end{thebibliography}

\end{document}